\documentclass[a4paper, 11pt]{article}
\pdfoutput=1
\usepackage{latexsym}
\usepackage[fleqn]{amsmath}
\usepackage{amssymb}
\usepackage{amsthm}
\usepackage{ulem}
\usepackage{tikz}
\usepackage{tkz-euclide}
\usetikzlibrary{calc,intersections}
\usepackage[dvips,a4paper, hmargin={3cm,3cm},vmargin={2.8cm,2.8cm}]{geometry}
\usepackage{enumitem}
\setlist{itemsep=0.4pt,topsep=0pt}
\usepackage{hyperref}
\usepackage{comment}

\DeclareMathOperator{\bl}{b}
\DeclareMathOperator{\cl}{cl}

\DeclareMathOperator{\acl}{acl}
\DeclareMathOperator{\tp}{tp}
\DeclareMathOperator{\Th}{Th}
\newcommand{\Cor}{\mathcal{C}_\mathrm{or}}
\newcommand{\Cun}{\mathcal{C}_\mathrm{un}}
\newcommand{\mons}{\mathbb{M}}
\newcommand{\strong}{\mathrm{strong}}

\newtheoremstyle{note}
{8pt}
{14pt}
{\normalfont}
{}
{\bfseries}
{ }
{.5em}
{}
\theoremstyle{note}
\newtheorem{thm}{Theorem}[section]
\newtheorem{lemma}[thm]{Lemma} 
\newtheorem{df}[thm]{Definition}
\newtheorem{rmk}[thm]{Remark} 
\newtheorem{prop}[thm]{Proposition}

\newtheorem{cor}[thm]{Corollary}

\newtheorem{ex}[thm]{Example}

\newtheorem*{notation}{Notation} 
\title{The elementary theory of free Steiner triple systems}

\author{Silvia Barbina,  Enrique Casanovas \thanks{Research partially supported by grants PID2020-116773GB-I00 and PID2023-147428NB-I00 from the Spanish Government. The first author has also been supported by PRIN2022 \textit{Models, sets and classifications}, prot. 2022TECZJA.}}

\begin{document}
	\maketitle

		\begin{abstract}
Free Steiner triple systems (STS) are infinite structures that are naturally characterised by a universal property. We consider the class of free STSs from a model theoretic viewpoint. We show that free STSs on any number of generators are elementarily equivalent. We axiomatise their theory and show that it is stable.
	\end{abstract}
	
	\section{Introduction}
	
	The subject of this paper is the first-order theory of a class of combinatorial structures called free Steiner triple systems. 	
	A \textbf{Steiner triple system} (STS) of order $\lambda$ is a pair $(M, \mathcal{B})$ where $M$ is a set with $\lambda$ elements, and 
 $\mathcal{B}$ is a collection of 3-element subsets of $M$ (the \textbf{blocks})
such that any two $x, y \in M$ are contained in exactly one block. 
A set $M$ with a collection of 3-element subsets is a \textbf{partial STS} (PSTS) if any two elements of $M$ belong to at most one block.

 Varying the parameters 2 and 3 gives the more general notion of a Steiner $(t,k)$-system, where blocks have size $k$ and every $t$-subset is contained in exactly one block.

Steiner systems are well known structures, widely studied in the finite case. Recently, certain infinite incarnations have appeared in the model-theoretic literature, e.g. as examples of strongly minimal structures \cite{baldwin-pao}, or as specific constructions such as Fra\"{\i}ss\'e amalgamations \cite{barcas}. These constructions exhibit a range of stability behaviours: the structure in \cite{barcas} is $\mathrm{TP}_2$ and $\mathrm{NSOP}_1$, and thus it lies at the opposite end of the stability spectrum from its counterparts in~\cite{baldwin-pao}.

An interest has emerged, more broadly, in the model theory of incidence structures --- see, for example, \cite{conant} for \textit{generic} incidence structures, \cite{hyttpao} for projective spaces, \cite{baldwin-pao} and \cite{bald22} for Steiner systems, \cite{ammertent} for generalised $n$-gons. In  \cite{ammertent,bald22,conant,hyttpao}, the structures involved are described as either two-sorted or bipartite, with one sort (respectively part) for points and one for lines. In \cite{baldwin-pao}, the vocabulary is relational and one-sorted.  As we shall see, here (and in~\cite{barcas})  Steiner triple systems are one-sorted structures, which means our methods are generally different.



A notable construction method that applies to incidence structures and, more generally, to algebras, generates \textit{free} structures.  Universal algebra gives a standard characterization of free constructions, which, in the case of Steiner triple systems, captures the intuition that an infinite STS can be obtained from an initial set $A$ containing $\lambda$ vertices with no blocks defined among them by adding, for each pair of vertices $a_1, a_2 \in A$, a new point that gives the unique block to which $a_1$ and $a_2$ belong. Iterating this process $\omega$ times gives an STS. The template example in \cite{chicotetal} is, essentially, the result of this construction for $\lambda=3$. 

Standard universal-algebraic tools tell us that such  free constructions are uniquely determined up to isomorphism by the cardinality of the generating set $A$. A series of natural questions then arises, such as which first-order properties, if any, are shared by two free constructions where the sizes of the initial generating sets are different.

Questions of this kind have been addressed, for example, for projective planes \cite{hyttpao} and for generalised polygons \cite{ammertent}. Towards the completion of the write-up of our results, we became aware of \cite{paoquad} --- see the end of this section for the connection with our work.

In these papers, free structures are characterised via the existence of specific constructions in stages, rather than through the universal algebraic definition. Our choice of a one-sorted, functional language (see below) for Steiner triple systems makes the connection with universal algebra rather natural, and distinguishes our methods.

A  partial Steiner triple system can be regarded as a first-order structure in the language containing a ternary relation symbol $R$ and satisfying the axioms
\begin{itemize}
\item for all permutations $\sigma \in S_3$, $ \forall x_1 x_2 x_3 \, \left(R(x_1, x_2, x_3) \rightarrow R(x_{\sigma(1)}, x_{\sigma(2)}, x_{\sigma(3)} \right)$
\item $\forall x y z \left( R(x,y,z) \rightarrow y\neq z \right)$
\item $\forall x y w z \left( R(x,y,w) \wedge R(x,y,z) \rightarrow w=z \right)$.
\end{itemize}
The relation $R$ is interpreted as the set of blocks: three elements $x, y$ and $z$ are in a block if and only if $R(x,y,z)$.  A partial STS that satisfies the additional axiom
\begin{itemize}
\item $\forall x y \left( x \neq y \rightarrow \exists z \,R(x,y,z)\right)$
\end{itemize}
is an STS.
With this choice of language, a substructure of an STS is a partial STS.

An STS can also be described as a \textbf{Steiner quasigroup}, that is, a structure with a single binary operation symbol $\cdot$, interpreted as an operation that satisfies the following axioms:
 \begin{enumerate}
\item $\forall x y \, (x \cdot y= y\cdot x)$
\item $\forall x \, (x\cdot x =x)$
\item $\forall x y \, (x\cdot (x\cdot y)=y)$.
\end{enumerate}
Thus, Steiner quasigroups are an equational class --- see \cite{ganter1, ganter2}. This is what makes it possible to employ universal algebraic tools.

A quasigroup $M$ determines an STS whose blocks are the elements of the graph of the operation $\cdot$ without the diagonal, that is,
\[ \mathcal{B} = \{ \{x,y,z \} \subseteq M : x \cdot y = z \} \smallsetminus \{ \{x \} : x \in M \} \, .\]
It is easy to see that each STS determines a quasigroup and each quasigroup determines an STS, where
  $$x \cdot y = z \mbox{ iff } R(x,y,z) \mbox{ or } x=y=z.$$  In fact, as  noted in \cite{ganter1, ganter2}, this shows that each Steiner triple system is bi-interpretable with a Steiner quasigroup, and justifies our choice to use the term \textit{Steiner triple system} to cover both the functional and relational structures from the end of this section onwards.

 A detailed account of the relation between Steiner quasigroups and Steiner systems can be found in \cite{baldwin23}, 
where the author shows that each strongly minimal Steiner $(2,k)$-system constructed in \cite{baldwin-pao} can be \textit{coordinatised} by a quasigroup when $k$ is a prime power. In particular, in Section~3 Steiner triple systems are shown to be bi-interpretable with, rather than simply coordinatised by, Steiner quasigroups. Instead, when $k \neq 3$ the strongly minimal Steiner $(2,k)$-systems constructed in~\cite{baldwin-pao} can be interpreted in appropriate quasigroups, but no interpretation can be found in the other direction. 
This confirms that treating Steiner \textit{triple} systems as quasigroups makes algebraic tools available that are not applicable when they are described as incidence structures.

Since the quasigroup axioms are universal, it is clear that a substructure of a Steiner quasigroup is a Steiner quasigroup. Moreover, if $N$ is a Steiner quasigroup and  $A \subseteq N$, there is a natural notion of substructure of $N$  \textbf{generated} by $A$, namely
\[ \langle A \rangle_N = \bigcap \{ M\subseteq N : M \mbox{ is a substructure that contains } A \} \,.\]
The subscript $N$ may be omitted 
when clear from the context.

\begin{rmk}\label{May12_1} Let $M$ be a Steiner quasigroup generated by $A\subseteq M$.  Then
\begin{itemize}
\item for every $b \in M$ there is a term $t(\bar{x})$ such that $b=t(\bar{a})$, where $\bar{a}$ enumerates $A$
\item $M=\bigcup_{n<\omega} S_n$ for a chain of sets $S_n$ such that  $S_0 =A$ and $S_{n+1}= \{a\cdot b\mid a,b\in S_n\}$. Note that $S_n\subseteq S_{n+1}$  since  $a\cdot a = a$.
\end{itemize}
	\end{rmk}
	
	\begin{notation} Below we write $S_{<n}$ for $\bigcup_{i<n} S_i$. \end{notation}

	\begin{df} Let $M$ be a Steiner quasigroup generated by $A\subseteq M$. The sets $S_n$ in Remark~\ref{May12_1} are called the \textbf{levels} of $M$ over $A$. An element $a\in M$ has   \textbf{level $n$ over $A$}  if $a\in S_n\smallsetminus S_{<n}$. 
	\end{df}

As remarked above, in what follows the term \textit{Steiner triple system} is used for Steiner quasigroups, too; we rely on the context to clarify the usage. 

In Section~\ref{section2} we characterise free STSs as those that are \textit{freely generated} by a subset in the sense of universal algebra, and we recall some standard universal-algebraic results. Section~\ref{section3} introduces certain linear orderings for STSs and partial STSs that are an essential tool in our treatment of the theory of free STSs. Moreover, we show that the existence of an \textit{HF-ordering} on an STS is equivalent to its \textit{unconfinedness}, a first-order condition that is satisfied by all infinite free STSs and will essentially axiomatise their theory. In Section~\ref{section4} we show that the amalgamation property fails for the class $\mathcal{C}_\mathrm{un}$ of finite unconfined partial STSs. We also show that the obvious choice of a predimension function does not yield a fruitful definition of \textit{strong substructure} for the construction of a Hrushovski amalgam. An appropriate notion of strong substructure is formulated in Section~\ref{section5} and it involves HF-orderings. We show that the countable free STS with $\omega$ generators is the generic object for $\mathcal{C}_\mathrm{un}$ with respect to this notion. We also give appropriate definitions of \textit{$n$-strong} substructure and \textit{$n$-amalgamation} and use the setup in \cite{wagner} to show that the generic is not saturated. In Section~\ref{section6} we show that infinite unconfined STSs always contain certain binary trees that are an essential ingredient 
in the proof of our main result in Section~\ref{section7}, where we prove that the theory of infinite unconfined STSs is complete and it is the theory of infinite free STSs --- so, in particular, any two free infinite Steiner triple systems are elementarily equivalent. We also show that the theory is stable and that it has a model which is not free. 

Our elementary equivalence result appears to be a special case of the main theorem in~\cite{paoquad}, developed independently, where the completeness and strict stability are asserted for the theories of a range of open incidence structures that includes free Steiner systems (hence free Steiner triple systems). The results in~\cite{paoquad} generalise those in~\cite{ammertent} and~\cite{hyttpao}. As previously noted, a two-sorted language is chosen in~\cite{paoquad}. This allows for a uniform treatment of the range of incidence structures considered in the paper. Since here we restrict our attention to Steiner triple systems, we can adopt the more natural choice of a one-sorted language, which leads to different methods and considerable simplifications.


	\section{Free Steiner triple systems: universal algebra}
 \label{section2}
We give a universal algebraic characterisation of free Steiner triple systems  as those that satisfy a \textit{universal mapping property}. 
	
	 \begin{df} Let $\mathcal{C}$ be a class of Steiner triple systems. An STS $M$ generated by $A$ has the \textbf{universal mapping property} (UMP)  \textbf{over $A$ for $\mathcal{C}$} if for every $N \in \mathcal{C}$, every map
 \[ \varphi : A \to N \]
extends to a homomorphism $ \hat{\varphi} : M \to N $.

 \end{df} 
  A standard argument (see, for example, \cite{gratzer}, \S~24, Corollary~1), ensures that the extension $\hat{\varphi}$ is unique.
 
 \begin{df} If $M$ is an STS, a subset $A \subseteq M$ is \textbf{independent} if $\langle A \rangle_M$ has UMP  over $A$ for the class of all STSs.

$M$ is  said to be \textbf{freely generated} by $A$ if
\begin{enumerate}
\item[(i)] $A$ is independent
\item[(ii)] $\langle A \rangle_M = M$.
\end{enumerate}

If $M$ is freely generated by $A \subseteq M$ then $M$ is a \textbf{free STS}, and $A$ is a \textbf{free base} of $M$.

\end{df}

	The next proposition formalises the intuition that a Steiner triple system freely generated by a subset $A$ can be constructed from $A$ in stages by adding, at each stage, a new vertex as the product of two existing vertices whose product is still undefined. This construction was first described in \cite{sieben} as the result of a \textit{free extension process}, and later generalised in \cite{funk} to various classes of incidence geometries that include, among others, projective and affine planes, generalised $n$-gons, $k$-nets and Steiner $(t,k)$-systems.  For convenience, we state a version for Steiner triple systems that lists explicitly the properties we require in the rest of the paper.

	\begin{prop}\label{May12_2} Let  $M$ be a Steiner triple system and let $A\subseteq M$. The following are equivalent:
		\begin{enumerate}
			\item $A$ is a free base of $M$.
			\item $M$ is generated by $A$ with levels $(S_n\mid n<\omega)$, where $S_0=A$ and
			\begin{enumerate}
				\item if $a,b\in S_n\smallsetminus S_{<n}$ and $a\neq b$, then $a\cdot b\not\in S_n$.
				\item if $a\in S_{n+1}\smallsetminus S_n$, there is a unique pair $\{b,c\}\subseteq S_n$ such that $b\neq c$ and $a= b\cdot c$.
			\end{enumerate}
		\end{enumerate}
		Moreover, property (b) implies that if  $a\in S_{n+1}\smallsetminus S_n$, and $b, a\cdot b\in S_n$, then  $\{b,a\cdot b\}$ is the unique pair given by (b) for  $a$.
	\end{prop}
	\begin{proof} See~\cite{barcas2},  Proposition~3.1.
	\end{proof}

Notice that each set $S_{n+1}\smallsetminus S_n$ in the construction above generates a subsystem of $M$ that satisfies condition 2, so  $S_{n+1}\smallsetminus S_n$ is a free base for $\langle S_{n+1}\smallsetminus S_n \rangle$, and in particular an independent set. It follows that $M$ contains independent sets of finite unbounded cardinality. In fact, $M$ also contains an independent set of cardinality $\omega$ --- a proof for Steiner loops in \cite{markovski} is easily adapted to Steiner triple systems.
  
		\begin{df} Let $M$ be a Steiner triple system with a free base $A\subseteq M$. The sequence $(S_n\mid n<\omega)$ of levels described in Proposition~\ref{May12_2} will be called a \textbf{standard free construction of $M$ over $A$}.
	\end{df}

Standard techniques for any variety in universal algebra show that the cardinality of a free base $A$ determines the STS freely generated by $A$ up to isomorphism. 

\begin{prop}\label{May11_9.1} Let $M,N$ be Steiner triple systems freely generated by $A\subseteq M$ and $B\subseteq N$ respectively. Then $M\cong N$ if and only if $|A|=|B|$. In fact, every bijection between $A$ and $B$ can be extended  to an isomorphism between $M$ and $N$.
	\end{prop}
	
	 \begin{proof} This follows from Theorem I in \cite{fujiwara}, and Lemma 2.6, Theorem 1 and Theorem 2 in \cite{jontarski}. See \cite{barcas2}, Corollary~2.11 for a self-contained version of these arguments in the  case of Steiner triple systems.
	\end{proof}
	
The elementary equivalence of any two infinitely generated free Steiner triple systems follows from a standard argument for infinitely generated free algebras.

	\begin{lemma} \label{lemma2_6}  Let $M,N$ be free Steiner triple systems and assume $A$ is an infinite free base of $M$ that can be extended to a free base $B$ of $N$.  Then  $M\preccurlyeq N$.
	\end{lemma}
	\begin{proof} \cite{gratzer}, Chapter~6, \S 38, Theorem 4.
	
	\end{proof}
	
	\begin{thm} If $M,N$ are  free Steiner triple systems with infinite free bases, then    $M \equiv N$.
	\end{thm}
	\begin{proof}   \cite{gratzer}, Chapter~6, \S 38, Corollary to Theorem~5.
	\end{proof}

	\section{HF-orderings and confined configurations} \label{section3}
	
	The idea of \textit{free} and \textit{hyperfree extension processes} appears in the context of free projective planes in \cite{sieben} and has subsequently been generalised to a wider range of incidence structures that includes Steiner triple systems in \cite{funk}.  Such processes yield well-founded constructions of incidence structures from existing ones by adding new elements one at a time, with suitable restrictions on the number of incidences. In \cite{hyttpao}, the related notions of \textit{HF-} (hyperfree) and \textit{F-} (free) \textit{orderings} are formulated for projective planes to dispense with the well-foundedness assumption. These orderings have natural parallels for partial and complete Steiner triple systems respectively.  Their properties are proved below by adapting arguments that appear in \cite{funk, hyttpao} for two-sorted incidence structures.

	\begin{df}
		A \textbf{free ordering}  (or \textbf{F-ordering})   of a Steiner triple system $M$ \textbf{over} $A\subseteq M$ is a linear ordering $<$ of $M$ such that for every $a\in M$  there is a unique pair $\{a_1,a_2\}\subseteq M$ with $a=a_1\cdot a_2$,  and, for each $i$, either $a_i\in A$ or $a_i<a$. Notice that when $a\in A$, the only possibility is $a_1=a_2=a$.
	\end{df}
	
	In our discussion of linear orderings we use standard operations as defined, for instance, in~\cite{rosenstein}. In particular, if $(A,<_A)$ and $(B,<_B)$ are linear orderings and $A\cap B=\emptyset$, then their sum  is the linear order 
		\[(A\cup B, <)= (A,<_A) + (B,<_B) \, ,\] 
where $<$ is defined as $<_A\cup <_B \cup \, (A\times B)$.  
	
	\begin{rmk}\label{May12_12.0} Let $M$ be a Steiner triple system with standard free construction \linebreak $(S_n\mid n<\omega)$ over $A$. Let $<_n$ be a linear ordering of $S_n\smallsetminus S_{<n}$ and for $a,b\in M$ define
		$$a<b \mbox{ iff for some } n,\; b\in S_n\smallsetminus S_{<n} \mbox{ and either } a\in S_{<n} \mbox{ or } a\in S_n\smallsetminus S_{<n} \mbox{ and } a<_n b.  $$
Then $<$ is an F-ordering of $M$ over $A=S_0$, which is an initial segment, and $<$ is in fact the generalized sum  $\sum_{n<\omega}(S_n\smallsetminus S_{<n}, <_n)$. If every $<_n$ is a well-order, $<$ is a well-order.
		
		\end{rmk}

	\begin{rmk}\label{May12_12.1}
		 Let $M$ be a Steiner triple system. If there is an F-ordering of $M$ over $A$, there is another one where $A$ is an initial segment. If the original order is a well-order, so is the new order,  and $M=\langle A\rangle$.
	\end{rmk}
	\begin{proof}
		If $<$ is an F-ordering of $M$ over $A$, then the sum $(A,<\restriction A)+(M\smallsetminus A, <\restriction (M\smallsetminus A))$ defines an F-ordering of $M$ over $A$ with initial segment $A$.
	\end{proof}	
	
	\begin{df} A \textbf{hyperfree ordering}  (or \textbf{HF-ordering}) of a partial Steiner triple system  $A$ is a linear ordering $<$ of $A$ such that for every $a\in A$ there is at most one pair of smaller elements $\{a_1,a_2\}\subseteq A$ with $a=a_1\cdot a_2$,  that is, there is at most one block $\{a,a_1,a_2\}$ with $a_1,a_2<a$.
	\end{df}

The next proposition explains the connection between free bases and F-/HF-orderings.

	\begin{prop}\label{May12_2.2}Assume $M$ is a Steiner triple system and $A\subseteq M$. The following are equivalent:
		\begin{enumerate}
			\item $A$ is a free base of $M$ (as a Steiner quasigroup)
			\item There is a well-founded F-ordering of $M$ over $A$
			\item There is a well-founded HF-ordering of $M$ and $A$ is the set of points that are not a product of smaller elements.
		\end{enumerate}
		\end{prop}
		\begin{proof}  See~\cite{barcas2}, Proposition~3.7.
		\end{proof}

	Free Steiner triple systems satisfy a property which can be expressed through infinitely many first-order universal sentences. 
	
				\begin{df} 
 A finite partial STS  $A$ is \textbf{confined} if every element of $A$ belongs to at least two blocks.

A partial STS 
 $A$ is \textbf{unconfined}, or \textbf{contains no confined configuration}, if no finite $A^\prime \subseteq A$ is confined.

Equivalently,  $A$ is unconfined if for all finite  nonempty $A^\prime \subseteq A$ there is $a \in A^\prime$ that belongs to at most one block in $A^\prime$, that is, $a$ is either not a product of any elements in $A^\prime$, or the product of exactly one pair of elements of $A^\prime$.
\end{df}
Note that being unconfined is stronger than not being confined: an unconfined partial STS is hereditarily not confined. 
	
	In \cite{hall}, \textit{confined} partial projective planes are those where every line contains three or more points, and there are three or more lines through every point. If a finite partial plane is not confined, then it is \textit{open}. These notions are later used in \cite{sieben} and they are central in  \cite{hyttpao}. Analogues are defined in \cite{funk} for the range of incidence structures treated in the paper, including Steiner systems. 
Our definition of unconfined partial STS is equivalent to that of \textit{open} partial STS in \cite{funk}.
	
	\begin{prop}\label{unconfined} \begin{enumerate}
			\item Free Steiner triple systems are unconfined.
			\item Finitely generated unconfined STSs are free.
			\item A partial STS is unconfined if and only if it has an HF-ordering.
		\end{enumerate}
	\end{prop}
	\begin{proof} See~\cite{barcas2}, Propositions~3.10, 3.13 and 3.11 for (1), (2) and (3) respectively.
	\end{proof}

\begin{df}	Let   $A\subseteq B$ be partial STSs. An HF-ordering $<$ of $B$ is \textbf{over} $A$ if  for every $a\in B\smallsetminus A$ there is at most one pair $\{a_1,a_2\}$ of distinct elements of $B$ such that $\{a,a_1,a_2\}$ is a block and  for every $i$, either $a_i<a$ or $a_i\in A$. 
	\end{df}
	
	Note that if $B$ is an STS, any F-ordering  of $B$ over $A\subseteq B$  is an HF-ordering over $A$, but even if an HF-ordering over $A$ is an F-ordering (over some subset of $B$), it is not necessarily an F-ordering over $A$, since elements of $A$  can be a product of smaller elements of $A$.
	
\begin{lemma}\label{order-over}
	Let $A\subseteq B$ be partial STSs. There is an HF-ordering of $B$ over $A$ if and only if there is an HF-ordering of $B$ where $A$ is an initial segment. In fact, given any HF-ordering $<$ of $B$ over $A$, the sum $(A,<\restriction A)+(B\smallsetminus A, <\restriction (B\smallsetminus A))$ is an HF-ordering of $B$ with initial segment $A$.
\end{lemma}	
\begin{proof}
	On the one hand, if $<$ is an HF-ordering of $B$ with initial segment $A$, it easy to see that it is an HF-ordering over $A$. On the other hand, assume $<$ is an HF-ordering of $B$ over $A$. Then the only way in which the sum can fail to be an HF-ordering is for some $a\in B \smallsetminus A$ to be a product of two different pairs of smaller elements, 
and each element is either smaller than $a$ in $<$ or belongs to $A$, contradicting that $<$ is an HF-ordering over~$A$.
\end{proof}
	
\begin{df} \label{dfclosure} Given an HF-ordering  $<$ of the partial STS $B$  and a subset $A\subseteq B$, we define the $<$-\textbf{closure} of $A$ (in $B$) as the smallest subset $C$ of $B$ containing $A$ and such that  whenever  $b\in C$ and  $b_1,b_2 < b$ and $\{b,b_1,b_2\}$ is a block, then  $b_1,b_2\in C$.  We denote~$C$ by  $\cl_<(A)$.  
	\end{df}  
It is clear that the $<$-closure is a finitary closure operator on subsets of $B$. Moreover,
\[\cl_<(A_1\cup A_2)= \cl_<(A_1) \cup \cl_<(A_2) \,.\]

\begin{lemma}\label{Feb_23_2024_7} If $<$ is an HF-ordering of $B$  and $A=\cl_<(A)\subseteq B$, then  $<$ is an HF-ordering of $B$ over $A$. 
\end{lemma}
\begin{proof}   If not, there is some $b\in B\smallsetminus A$ and there are distinct $b_1,b_2,b_3,b_4$, such that  $\{b,b_1,b_2\}$,  $\{b,b_3,b_4\}$  are  blocks and, for each $i$, either $b_i < b$ or $b_i \in A$.   Without loss of generality,  $b_2 <b_1$. We claim that $b_1 < b$.  If not, then   $b_1\in A$ and $b <b_1$,  and since $A=\cl_<(A)$, we get $b \in A$, a contradiction. Hence,  $b>b_1>b_2$. By a similar argument, $b> b_3,b_4$.  But then $<$ is not an HF-ordering.
\end{proof}

\begin{lemma}\label{June_25_2024_1.1} Assume $<$ is an HF-ordering of $B$,  $A=\cl_<(A)\subseteq B$  and  $C\subseteq B$ is an initial segment  such that  $c<a$  whenever $c\in C$ and $a\in A$. Then we can interpolate $A$ between $C$ and the rest of the order of $B$ in such a way that we obtain an HF-ordering. In other words,   $$(C, <\restriction C) +(A, <\restriction A) + (B\smallsetminus (A\cup C), <\restriction B\smallsetminus(A\cup C))$$ is an HF-ordering of $B$. Moreover, we can change the ordering
 of $A$ by any other HF-ordering of $A$ and still have an HF-ordering of $B$.
\end{lemma}
\begin{proof} Notice that $(C, <\restriction C) +(A, <\restriction A)= (A\cup C,<\restriction (A\cup C))$. Since $\cl_<(A\cup C)=A\cup C$,  Lemma~\ref{Feb_23_2024_7}  ensures that $<$ is an HF-ordering of $B$  over $A\cup C$.  For the moreover clause,  note that no element of $A$ is a product of elements of $C$  nor a product of an element of $C$ by an element of $A$.
\end{proof}

\begin{lemma}\label{order_closure}
	If $<$ is an HF-ordering of $M$,  $A\subseteq M$  and $\cl_<(A)=A$, then  $\cl_<(\langle A\rangle )= \langle A \rangle$.
\end{lemma}
\begin{proof}
	Let $\langle A\rangle = \bigcup_{n<\omega}A_n$, where $A_0=A$  and $A_{n+1}$ is the set of elements  of $\langle A\rangle\smallsetminus (A_0\cup \ldots \cup A_n)$  which are a product of two elements of $A_0\cup \ldots 	\cup A_n$.  It is enough to prove by induction that, for every $n$,   $\cl_<(A_0\cup \ldots \cup A_n)= A_0\cup\ldots\cup A_n$. This is clear for $n=0$. Assume $a\in A_{n+1}\smallsetminus (A_0\cup\ldots \cup A_n)$. There are $b,c\in A_0\cup \ldots \cup A_n$  such that $a=b\cdot c$. If $b>a$ or $c>a$, then $a\in \cl_<(A_0\cup \ldots \cup A_n)= A_0\cup \ldots \cup A_n$, a contradiction. Otherwise, $a>b,c$ and $\{b,c\}\subseteq A_0\cup\ldots\cup A_n$ is the unique pair of elements such that $a>b,c$ and $a=b\cdot c$.
\end{proof}

\section{Some remarks on predimension} \label{section4}

In \cite{hall}, a notion of \textit{rank} for projective planes is defined whose analogue for Steiner triple systems coincides with a natural choice for a predimension function. As we shall see, though, this predimension function is not suitable for defining a notion of strong substructure that would lead to the construction of a Hrushovski-style amalgam --- in fact, this is what motivates the different choice of predimension to build the uncountably many strongly minimal Steiner systems in~\cite{baldwin-pao}.

We show that the amalgamation property fails in the class of finite partial unconfined STSs with respect to both the substructure relation, and the obvious definition of strong substructure via predimension. This is why, in Section~\ref{section5}, our construction of a generic object requires a rather different, and possibly not as natural, notion of strong substructure. 

We also include a comparison with a class of finite partial STSs, defined in \cite{evans} in a wider context, that properly contains the unconfined ones and where predimension works well instead. 
While these results have no direct application in subsequent sections, they clarify the context and may have independent interest.

\begin{df} \label{dfpredim} Let $A$ be a finite partial STS and let $\bl(A)$ be the set of blocks in $A$. Then the \textbf{rank}, or \textbf{predimension}, of $A$ is
\[ \delta(A)= |A|-|\bl(A)| \, .\]
\end{df}

\begin{rmk} It is well known  (see, for example, Subsection~1.1.1 in \cite{colbournrosa}) that in a finite STS $M$ of cardinality $n$
\begin{itemize}
\item any point lies in $\frac{n-1}{2}$ blocks
\item there are $\frac{n(n-1)}{6}$ blocks.
\end{itemize}

It follows that 
\begin{enumerate}
\item an unconfined STS $M$ is infinite, or $|M|=3$
\item if $M$ is  a finite STS of cardinality $> 3$ then  $\delta(M)\leq 0$.
\end{enumerate}

For 1, let $M$ be a finite STS such that $|M| > 3$. Then $M$ is a partial STS contained in $M$ such that any $p \in M$ lies in $\frac{n-1}{2}$ blocks. Since $n>3$, $\frac{n-1}{2}>1$.

For 2, let $M$ be a finite STS of cardinality $n \geq 3$. Then 
\[ \delta(M) = n - \dfrac{n(n-1)}{6} \]
and $\delta(M) \geq 0$ if and only if $3 \leq n \leq 7$. In fact, the predimension of the Fano plane (the unique STS of order 7) is 0, and the predimension of all other finite, non-trivial STS is negative.

Nevertheless, suitable subclasses of the class of finite partial STSs  can be characterised in terms of non-negative dimension. The next lemma is an adaptation of the corollary to Lemma~1 in \cite{sieben}. 
\end{rmk}

\begin{lemma} The predimension of a finite partial STS $A$ that contains no confined configuration is non-negative. \end{lemma}
\begin{proof} 
Suppose $|A|=n$. Since $A$ is unconfined, we can find an HF-ordering 
\[a_1 < a_2 < \ldots < a_n \, .\]
Let $A_i:= \{a_1, \ldots, a_i \}$ for $1 \leq i \leq n$. Then, if  $|\mathrm{b}(A_i)|$ is the number of blocks in $A_i$,
$$ \delta(A_i)= \begin{cases} |A_{i-1}|- |\mathrm{b}(A_{i-1})| +1 -1 = \delta(A_{i-1}) & \mbox{ if $a_i$ belongs to one block in } A_i \\
|A_{i-1}|- |\mathrm{b}(A_{i-1})| +1 & \mbox{ if $a_i$ belongs to no blocks in } A_i \, .\end{cases} $$
Therefore $\delta(A_i) \geq \delta(A_{i-1})$ for all $i$. 
\end{proof}

\begin{rmk} The function $\delta$ is a {\bf predimension} in the following sense:
\begin{enumerate}
\item $\delta(\emptyset)=0$
\item $\delta(\{a\})\leq 1$  (in fact,  in our case $= 1$)
\item $\delta(A\cup B)+\delta(A\cap B)\leq  \delta(A)+\delta(B)$  if  $A,B$ are subsystems of a common partial STS $P$ ({\bf Submodularity} of $\delta$)\,.
\end{enumerate}
\end{rmk}

This predimension function can be used to define a notion of well-embedded substructure in the class of partial STS.

\begin{df} If $A$ is a finite partial STS and $B\supseteq A$ is a partial STS (possibly infinite), we define
$$A\leq^* B \Leftrightarrow  \delta(A)\leq \delta(B_0) \mbox{ for every finite }  B_0 \mbox{ such that } A\subseteq B_0\subseteq B.$$
If $A\leq^* B$, we say that $A$  is {\bf well embedded} in $B$. 
\end{df}

We compare the behaviour of the predimension $\delta$ in the class of finite unconfined partial STS and the class of \textit{oriented} partial STS  defined in \cite{evans} in the more general case of finite oriented colored $k$-hypergraphs that have a $\mathcal{G}$-orientation. The class defined below is the case $k=3, t=2$ and a single colour in Definition~1.3 in~\cite{evans}.

\begin{df} 
An {\bf orientation} of a partial STS $P$  is a one-to-one mapping $c:\bl(P)\rightarrow P$  such that  $c(e)\in e$ for every block $e$, where $\bl(P)$ is the set of blocks of $P$ as in Definition~\ref{dfpredim}. 

We say that $c(e)$ is an {\bf apex} of $e$. Given an orientation $c$ of $P$, we say that $A\subseteq P$ is {\bf $c$-closed} in $P$ or {\bf closed with respect to} $c$  if for every block $e$ of $P$ with apex in $A$, $e\subseteq A$.

Let  $\mathcal{C}$ be the class of all finite partial STSs. Then 
\begin{itemize}
\item $\Cor =\{A\in\mathcal{C}\mid A \mbox{ can be oriented and for every } X\subseteq A \mbox{ with } |X|\leq 2,\; X\leq^\ast A\} \, $
\item $ \Cun=\{A\in\mathcal{C}\mid A \mbox{ is unconfined }\} \,. $
\end{itemize}
\end{df}

We show that the finite unconfined partial STS are a proper subclass of  $\Cor$. 
\begin{prop} $\Cun\subseteq \Cor \, .$
\end{prop}
\begin{proof}  We check by induction on $|A|$ that if $A$ is unconfined, then it has an orientation.  We may assume $A\neq \emptyset$. Let $a\in A$ belong to at most one block of $A$. By induction hypothesis, $A\smallsetminus \{a\}$ has an orientation $c$. If $a$ belongs to the block $e$, we make $a$ the apex of $e$. This gives an orientation of $A$ extending $c$.

Now we prove by induction on  $|A|$ that if $A$ is unconfined and  $X\subseteq A$, and $|X|\leq 2$, then $X\leq^* A$.  It is enough to check that if $X\subseteq A$, then $\delta(X)\geq \min\{ |X|,2\}$.  We may assume $A\neq\emptyset$. Choose $a\in A$ belonging to at most one block of $A$ and let $A^\prime= A\smallsetminus\{a\}$. We can apply the induction hypothesis to $A^\prime$. If $a\not\in X$, then $X\subseteq A^\prime$ and the induction hypothesis gives the result. If $a\in X$, let  $X^\prime =X\smallsetminus\{a\}$. Then  $\delta(X^\prime)=\delta(X)$ or  $\delta(X)=\delta(X^\prime)+1$.  If $|X|\leq 2$, then $\delta(X)=|X|$ and clearly $\delta(X)\geq \min\{|X|,2\}$.  If $|X|>2$, then 
$|X^\prime| \geq 2$, and by induction hypothesis $\delta(X^\prime) \geq \min \{|X^\prime|,2 \}=2 $, hence $\delta(X) \geq  \delta(X^\prime) \geq 2= \min \{|X|, 2 \} $.
\end{proof}
 The next example shows that $\Cun\neq \Cor$.
\begin{ex}  Let $A$ be the partial STS consisting of the vertices  $\{a_1,\ldots,a_9\}$  with horizontal  blocks  
\[ \{a_1,a_2,a_3\},\{a_4,a_5,a_6\},\{a_7,a_8,a_9\} \]
 and  vertical blocks 
 \[\{a_1,a_4,a_7\},\{a_2,a_5,a_8\},\{a_3,a_6,a_9\} \,. \]
 Then $A$ is not unconfined, but it can be oriented and it satisfies the additional condition on subsets $X$ with $|X|\leq 2$. Hence, $A\in\Cor\smallsetminus \Cun$.
\end{ex}

By Corollary~1.1 in \cite{evans}, the class ($\Cor, \leq^*)$ has the amalgamation property. However, the following example shows that $\leq^*$ does not behave equally well in $\Cun$, which renders the results in \cite{evans} not applicable in our case.

 \begin{ex}
 \label{noamalgam}
 Consider two isomorphic partial STS 
 \begin{itemize}
 \item $B_1$ on the set $\{a,b,e,f,\, a \cdot b, \, a \cdot e, \, e \cdot f, \, f \cdot (a \cdot e), \, b \cdot (a \cdot e), \, f \cdot (b \cdot (a \cdot e)) \}$
 \item $B_2$ on the set $\{a^\prime, b^\prime, c,d,\, a^\prime \cdot b^\prime, \, a^\prime \cdot c, \, c \cdot d, \, c \cdot (a^\prime \cdot b^\prime), \, d \cdot (a^\prime \cdot c), \, d \cdot (c \cdot (a^\prime \cdot b^\prime)) \}$
 \end{itemize}
 with blocks as in the figure below.
 
 Let $A_1 =\{a,b \} \subseteq B_1$ and $A_2 =\{a^\prime,b^\prime \} \subseteq B_2$. Then the map $ \alpha = \{ \langle a, a^\prime \rangle, \langle b, b^\prime \rangle \}$ is an isomorphism between $A_1$ and $A_2$.
 
 
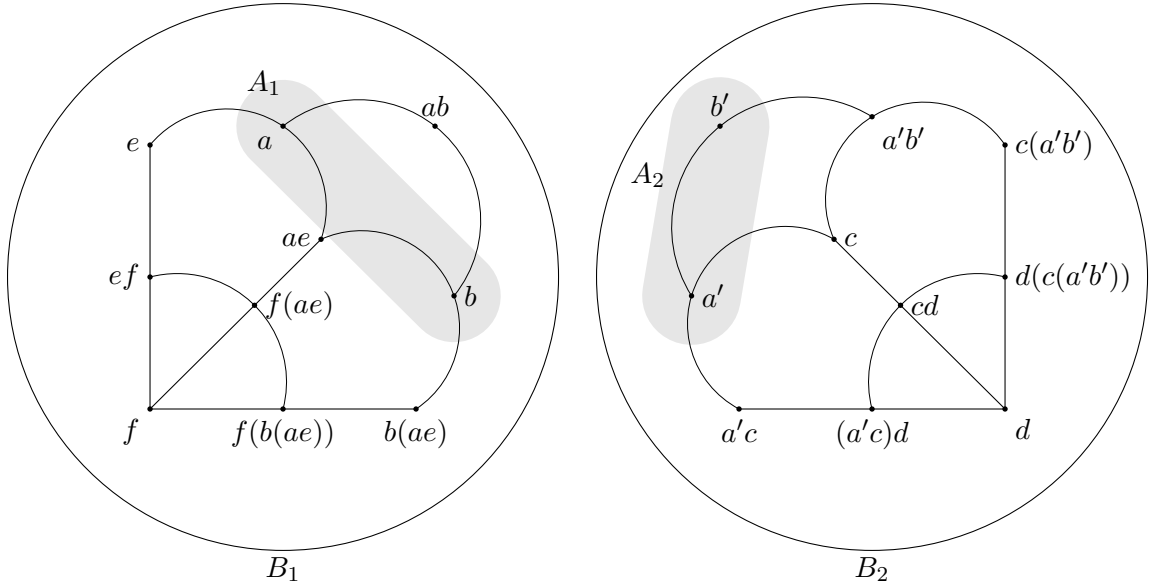
\begin{tikzpicture}[scale=0.25]
\tkzDefPoint[label=left:$e$](0,14){e}
\tkzDefPoint[label=below left:$a$](7,15){a}
\tkzDefPoint[label=left:$ae$](9,9){ae}
\fill (e) circle (4pt);
\fill (a) circle (4pt);
  \fill (ae) circle (4pt);
\tkzCircumCenter(e,a,ae)\tkzGetPoint{O}
\tkzDrawArc(O,ae)(e)

\tkzDefPoint[label=below left:$f$](0,0){f}
\tkzDefPoint[label=below:$f(b(ae))$](7,0){g}
\tkzDefPoint[label=below:$b(ae)$](14,0){h}
\fill (f) circle (4pt);
\fill (g) circle (4pt);
\fill (h) circle (4pt);
\draw (f)--(g)--(h);

\tkzDefPoint[label=left:$ef$](0,7){ef}
\fill (ef) circle (4pt);
\draw (f)--(ef)--(e);

\tkzDefPoint[label=right:$f(ae)$](5.5,5.5){fae};
\fill (fae) circle (4pt);
\draw (f)--(fae)--(ae);

\tkzDefPoint[label=right:$b$](16,6){b};
\fill (b) circle (4pt);
\tkzCircumCenter(ae,b,h)\tkzGetPoint{P}
\tkzDrawArc(P,h)(ae)

\tkzCircumCenter(ef,fae,g)\tkzGetPoint{Q}
\tkzDrawArc(Q,g)(ef)

\tkzDefPoint[label=above:$ab$](15,15){ab};
\fill (ab) circle (4pt);
\tkzCircumCenter(b,ab,a)\tkzGetPoint{R}
\tkzDrawArc(R,b)(a)

\draw [line width=35pt,line cap=round,opacity=0.1,rounded corners] (a.center) -- (b.center);
\node at (6,17.3) {$A_1$};

\draw (7,7) circle (14.5cm);
\node at (7,-8.5) {$B_1$};


\begin{scope}[xscale=-1,xshift=-45cm]
\tkzDefPoint[label=right:$c(a^\prime b^\prime)$](0,14){e}
\tkzDefPoint[label=below right:$a^\prime b^\prime$](7,15.5){a}
\tkzDefPoint[label=right:$c$](9,9){ae}
\fill (e) circle (4pt);
\fill (a) circle (4pt);
  \fill (ae) circle (4pt);
\tkzCircumCenter(e,a,ae)\tkzGetPoint{O}
\tkzDrawArc(O,ae)(e)

\tkzDefPoint[label=below right:$d$](0,0){f}
\tkzDefPoint[label=below:$(a^\prime c)d$](7,0){g}
\tkzDefPoint[label=below:$a^\prime c$](14,0){h}
\fill (f) circle (4pt);
\fill (g) circle (4pt);
\fill (h) circle (4pt);
\draw (f)--(g)--(h);

\tkzDefPoint[label=right:$d(c(a^\prime b^\prime))$](0,7){ef}
\fill (ef) circle (4pt);
\draw (f)--(ef)--(e);

\tkzDefPoint[label=right:$cd$](5.5,5.5){fae};
\fill (fae) circle (4pt);
\draw (f)--(fae)--(ae);

\tkzDefPoint[label=right:$a^\prime$](16.5,6){b};
\fill (b) circle (4pt);
\tkzCircumCenter(ae,b,h)\tkzGetPoint{P}
\tkzDrawArc(P,h)(ae)

\tkzCircumCenter(ef,fae,g)\tkzGetPoint{Q}
\tkzDrawArc(Q,g)(ef)

\tkzDefPoint[label=above:$b^\prime$](15,15){ab};
\fill (ab) circle (4pt);
\tkzCircumCenter(b,ab,a)\tkzGetPoint{R}
\tkzDrawArc(R,b)(a)

\draw [line width=37pt,line cap=round,opacity=0.1,rounded corners] (b.center) -- (ab.center);

\node at (18.8,12.4) {$A_2$};
\draw (7,7) circle (14.5cm);
\node at (7,-8.5) {$B_2$};
\end{scope}

\end{tikzpicture}

By inspection, we see that  
$A_i \leq^* B_i$ for $i=1,2$.

Amalgamating $B_1$ and $B_2$ over $A_2=\alpha(A_1)$ forces the identification $a \cdot b \mapsto a^\prime \cdot b^\prime$. The amalgam that results from this identification and no others is in $\Cor$ but not in $\Cun$, because all its vertices are in at least two blocks. It follows that $(\Cun, \leq^*)$ does not have the amalgamation property.

Moreover, in any other amalgam, for each vertex the valency (i.e. the number of blocks to which the vertex belongs) cannot be lower than the valency in the original partial STSs. It follows that any amalgam of $B_1$ and $B_2$ over $A_1 \cong A_2$ is not unconfined. Therefore this example shows, more generally, that the amalgamation property does not hold for the class $(\Cun, \subseteq)$, where $\subseteq$ is the substructure relation.
 \end{ex}

	\section{Amalgamation} \label{section5}
	
	Since, by Example~\ref{noamalgam}, the class $(\Cun, \subseteq)$ does not have the amalgamation property, we introduce a notion of strong substructure based on HF-orderings which makes $\Cun$ a generalised amalgamation class. This leads to the construction of a generic object for $\Cun$ (with respect to the notion of strong substructure) which turns out to be the free Steiner triple system on $\omega$ generators.
	
	\begin{df} \label{def5_1} For $A, B \in \Cun$,
$$A \leq B \  \Leftrightarrow \ A\subseteq B \mbox{ and there is an HF-ordering of $B$ where $A$ is an initial segment}. $$
When $A \leq B$, we say that $A$ is \textbf{strong} in $B$.
\end{df}

Since $A$ and $B$ are finite, the condition $A \leq B$ is first-order.

To see this, we define certain formulas in the relational language of Steiner triple systems (so blocks interpret a ternary relation $R$ that is symmetric and irreflexive). In what follows, we implicitly assume that the elements $a_i$ are pairwise distinct. There is a formula $\delta_1(x_1, \ldots, x_n)$, in the relational language of partial STSs, that says that $x_1<x_2<\ldots <x_n$ is an HF-ordering, namely
$$\delta_1(x_1, \ldots, x_n) \equiv \bigwedge_{i=1}^n \bigg[ \bigwedge_{1\leq j <k<i} \bigg( R(x_i, x_j, x_k) \rightarrow \bigwedge_{\tiny \begin{array}{l} 1\leq j' <k'<i \\ \{j,k\} \neq \{j',k'\} \end{array}} \neg R(x_i,x_{j'},x_{k'}) \bigg) \bigg] \, .$$

Now let $S_n$ denote the symmetric group on $n$ elements, and let
$$\delta_2 (x_1, \ldots, x_n) \equiv \bigvee_{\pi \in S_n} \delta_1(x_{\pi(1)}, \ldots, x_{\pi(n)}).$$
Then $\delta_2(a_1, \ldots, a_n)$ holds if and only if there is an HF-ordering of $a_1, \ldots, a_n$.

Consider the following modification of $\delta_2$:
$$\delta_3 (x_1, \ldots, x_n, x_{n+1}, \ldots , x_{n+m}) \equiv \bigvee_{\footnotesize \begin{array}{l} \pi \in S_n \\ \rho \in S_m \end{array}} \delta_1(x_{\pi(1)}, \ldots, x_{\pi(n)}, x_{n+\rho(1)}, \ldots, x_{n+\rho(m)}).$$
Then for $A=\{a_1, \ldots, a_n\} \subseteq \{a_1, \ldots, a_n, a_{n+1}, \ldots, a_{n+m}\}=B$ we have 
\[ B \models \delta_3 (a_1, \ldots, a_n, a_{n+1}, \ldots , a_{n+m}) \mbox{ if and only if } A\leq B \,. \]
So for each $m,n$, $A \subseteq B$ with  $|A|=n$ and $|B|=n+m$, there is a formula that asserts that $A \leq B$.

\begin{rmk}
	$\Cun$ is closed under isomorphism
	and $\leq$ is a partial order on $\Cun$ refining $\subseteq$ and satisfying  the following conditions:
	\begin{enumerate}
		\item If $A\leq B$ and $f:B\cong B^{\prime}$ is an isomorphism such
		that $f(A)=A^{\prime}$, then $A^{\prime}\leq B^{\prime}$.
		\item If $A\subseteq B\subseteq C$, $B\in \Cun$ and $A\leq C$, then $A\leq B$.
	\end{enumerate}
	This means that $(\Cun,\leq)$ is a \textit{smooth class} as defined in~\cite{kue-lask}.
\end{rmk}

The following are conditions C3 and  C4 in~\cite{wagner}.

\begin{lemma}\label{diamond} Let $A,A_1,A_2 \in \Cun$. Then
	\begin{enumerate}
		\item $\emptyset \leq A$  
	\item If $B\in \Cun$ is such that $A_1\subseteq B$ and $A_2\leq B$, then $A_1\cap A_2\leq A_1$.
	\end{enumerate}
\end{lemma}

\begin{proof} 1. By Proposition~\ref{unconfined},  since every $A\in \Cun$ has an HF-ordering.
	
	2. Given an HF-ordering of $B$ with initial segment $A_2$, the elements of $A_1$ inherit an HF-ordering with initial segment $A_1\cap A_2$.
\end{proof}

If $B$ is a partial unconfined STS, possibly infinite,   and $A \subseteq B$ is finite, we define
\[ A \leq B \ \Leftrightarrow \ A \leq A^\prime \mbox{ for every } A^\prime \in \Cun \mbox{ such that } A \subseteq A^\prime \subseteq B \, . \]
 The condition $A \leq B$ is first-order in this case, too: for $m \in \omega$, let $\delta_{4,m}(x_1, \ldots, x_n) $ be the formula that says 
$$\{x_1, \ldots, x_n\} \mbox{ is strong in any }m\mbox{-element extension }\{x_1, \ldots, x_n, y_1, \ldots, y_m\}\, ,$$
  that is
$$ \forall y_1 \ldots y_m \, \bigg[ \bigg( \bigwedge_{\tiny \begin{array}{l} 1 \leq i \leq n \\ 1 \leq j \leq m \end{array}} \neg x_i=y_j \ \ \wedge \bigwedge_{1\leq i <j \leq m} \neg y_i=y_j \bigg) \rightarrow \delta_3(x_1, \ldots, x_n, y_1, \ldots, y_m) \bigg] \,. $$
Now let 
$$\strong_m(x_1, \ldots, x_n) \equiv \bigwedge_{k\leq m} \, \delta_{4,k}(x_1, \ldots, x_n) \, ,$$
and suppose $A=\{a_1, \ldots, a_n \} \subseteq B$, where the $a_i$ are pairwise distinct. Then
$$B \models \left\{\strong_m(a_1, \ldots, a_n) : m\in \omega \right\} \mbox{ if and only if } A \leq B \, .$$

	 We shall see that 
$(\Cun, \leq)$ has the amalgamation property AP.  Now that we have shown that $(\Cun,\leq)$ is smooth, by \cite{kue-lask} this will imply that there is a countable unconfined Steiner triple system $F$, the \textbf{generic model}, such that
\begin{itemize}
\item $F = \bigcup_{n < \omega} C_n$ where $C_i \in \Cun$ and $C_i \leq C_j$ if $i\leq j$
\item if $A \in \Cun$, then there is an embedding $f:A \rightarrow F$ such that $f(A)\leq F$
\item for all $A \leq B$ such that $A \leq F$, there is an embedding $f: B \to F$ such that $f(B) \leq F$ and $f|_A = \mathrm{id}_A$.
\end{itemize}
Moreover, up to isomorphism there is a unique countable model satisfying these three properties.

 The class $\Cun$ turns out to be an interesting example of a variant of $\leq$-amalgamation appearing in~\cite{wagner} which we define next in our context and which allows us to establish that our generic is not saturated. We prove that this variant, called $n$-amalgamation, holds in $\Cun$. We subsequently use it to prove the amalgamation property.

\begin{df}  Let $A,B\in \Cun$ and $n<\omega$.   We write  $A\leq^n B$  if  $A \subseteq B$  and for every  $X\subseteq B\smallsetminus A$  such that  $|X|\leq n$, we have  $A\leq A\cup X$. 

Given  a partial unconfined STS $B$, possibly infinite, and $A\in \Cun$ such that   $A\subseteq B$, we write  $A\leq^n B$  if $A\leq^n B_0$ for every  finite $B_0\subseteq B$ . 

In either case we say that $A$ is \textbf{$n$-strong} in $B$. 
\end{df}
Note that $A\leq B$ if and only if $A\leq^n B$  for every $n$. 
\begin{rmk}\label{strongformula}
The condition of being an $n$-strong substructure of a partial unconfined STS $B$ is first-order: for $A=\{a_1, \ldots, a_k\} \subseteq B$, where the $a_i$ are pairwise distinct,
 $$B\models \strong_n(a_1,\ldots,a_k) \mbox{ if and only if } \{a_1,\ldots,a_k\}\leq^n B.$$
\end{rmk}

\begin{lemma}\label{sat3}  If  $A\leq^n B$ and $B\leq^n C$, then $A\leq^n C$.
\end{lemma}
\begin{proof}  Assume $X\subseteq C \smallsetminus A$  and $|X|\leq n$. We check that  $A\leq A\cup X$. Since $B\leq^n C$, we have $B\leq B\cup X$. By Lemma~\ref{diamond},   $A\cup (X\cap B)\leq A\cup X$.
	Since $A\leq^n B$, we know that $A\leq A\cup (X\cap B)$.  Hence  $A\leq A\cup X$.
\end{proof}

\begin{lemma}\label{Feb_23_2024_1}\begin{enumerate}
			\item Assume $A\leq B$, $ b\in B\smallsetminus A$ and there is a block $\{a_1,a_2,b\}$  with  $a_1,a_2\in A$. Then  $Ab\leq B$.
			\item Assume $A\leq^n B$, $b\in B\smallsetminus A$ and there is a block $\{a_1,a_2,b\}$  with  $a_1,a_2\in A$. Then  $Ab\leq^{n-1} B$.
		\end{enumerate}
	\end{lemma}
	\begin{proof} 1.  Start with an HF-ordering of $B$ with initial segment $A$. Then $\{a_1,a_2,b\}$ is the only block containing $b$ and elements  smaller than $b$ in the order.  We can make $b$ the least element in $B\smallsetminus A$ and obtain, again, an HF-ordering of $B$, now with $Ab$ as initial segment. 
		
		2. Similar proof.
	\end{proof}
	
\begin{df}
	We say that $(\Cun, \leq)$  has the \textbf{$n$-amalgamation property} if  for every $A\leq B$ in $\Cun$ there is some $m<\omega$  such that whenever $C\in \Cun$ and $A\leq^m C$, there are $D\in \Cun$ and embeddings $f:B\rightarrow D$ and $g:C\rightarrow D$   such that 
$$ f\restriction A= g\restriction A, \quad f(B)\leq^n D \, \mbox{ and } \, g(C)\leq D \, .$$
\end{df}

In~\cite{wagner} Wagner uses the term \textit{$\leq^\ast$-amalgamation} to mean that $n$-amalgamation holds for every $n$. The notion also appears in~\cite{baldwin-shi} with the name \textit{uniform amalgamation}. In~\cite{herwig}, Herwig proves that saturation of the generic is equivalent to $\leq^\ast$-amalgamation, but under the additional assumption that there is no infinite ascending chain of minimal pairs (condition C2 in~\cite{wagner}). By Lemma~2.18 in~\cite{baldwin-shi}, this means 
 that every model of the theory of the generic has \textit{finite closures} according to Definition~2.10 in~\cite{baldwin-shi}. Since our generic does not contain infinite ascending chains of minimal pairs, the saturation of the generic is in fact equivalent to the conjunction of $\leq^\ast$-amalgamation and C2. We show that C2 does not hold in our class $(\Cun,\leq)$, and so the generic is not saturated in our case.

\begin{df}
	A \textbf{minimal pair} is a pair $(A,B)$ of elements of $\Cun$ such that $A\subsetneq B$, $A\not\leq B$ and $A\leq A^\prime$ for every $A^\prime$ such that $A\subseteq A^\prime\subsetneq B$.
\end{df}

\begin{prop}
	$\Cun$ contains an infinite ascending chain of minimal pairs, that is, there is a sequence $((A_i,A_{i+1})\mid i<\omega)$ in $\Cun$  such that each $(A_i,A_{i+1})$ is a minimal pair. Hence condition C2 of~\cite{wagner} does not hold.
\end{prop}
\begin{proof} Let $F_3$ be the free STS with three generators $a_1,a_2,a_3$, and let $b=(a_1\cdot a_2)\cdot (a_1\cdot a_3)$. As shown in~\cite{barcas2}, the homomorphism $f:F_3\rightarrow F_3$  induced by $f(a_1)=b$, $f(a_2)=a_2$ and $f(a_3)=a_3$ is an embedding and it is not surjective. It is easy to see that if $A=\{b,a_2,a_3\}$ and $B=\{b,a_1,a_2,a_3, a_1\cdot a_2, a_1\cdot a_3\}$, then $(A,B)$ is a minimal pair. This construction can be iterated as follows.  Start with $N_0$ freely generated by $A_0=\{a_1,a_2,a_3\}$ and let $b_0= a_1$. Assume $N_n$ is freely generated by $\{b_n,a_2,a_3\}$  and find a proper extension $N_{n+1}$  freely generated by $\{b_{n+1},a_2,a_3\}$, where $b_n= (b_{n+1}\cdot a_2)\cdot(b_{n+1} \cdot a_3)$. Let $A_{n+1}= A_n\cup\{b_{n+1},a_2\cdot b_{n+1}, a_3\cdot b_{n+1}\}$. Then $(A_n,A_{n+1})$ is a minimal pair.
	
\end{proof}

	\begin{prop}\label{Feb_23_2024_2} $(\Cun,\leq)$ has the $n$-amalgamation property for every $n$.  In fact, if $A\leq B$ and $n<\omega$ are given, we can always take $m= n+|B\smallsetminus A|$.
	\end{prop}
	\begin{proof} Assume $A\leq B$  and  $A\leq^m C$, where $m= n+|B\smallsetminus A|$. We want to find $B^\ast\cong_A B$ and $D\in \Cun$  such that  $C\leq D$ and $B^\ast\leq^n D$.  We may assume $C\cap B=A$. Note that the blocks containing only elements of $A$ are the same in $B$ and in $C$.  Assume there are $b\in B\smallsetminus A$, $c\in C\smallsetminus A$ and  $a_1,a_2\in A$  such that $\{a_1,a_2,b\}$ is a block of $B$ and  $\{a_1,a_2,c\}$ is a block of $C$. Redefine $A^\prime= A\cup\{c\}$,  $B^\prime = (B\smallsetminus \{b\})\cup \{c\}$ and modify the blocks of $B$ by exchanging $b$ with $c$  to obtain the blocks of $B^\prime$. Then $B^\prime\cong_A B$ and $B^\prime\cap C=A^\prime$.  By Lemma~\ref{Feb_23_2024_1}, we have  $A^\prime\leq B^\prime$ and $A^\prime\leq^{m-1} C$. Iterating this process gives $B^\ast\cong_A B$  such that if $A^\ast =B^\ast \cap C$, then   $A^\ast \leq B^\ast$,  $A^\ast\leq^{n +i}C$  for some $i \geq 0$  and there are no $b\in B^\ast\smallsetminus A^\ast$, $c\in C\smallsetminus A^\ast$ and $a_1,a_2\in A^\ast$ such that  $\{a_1,a_2,b\}$ is a block of $B^\ast$  and  $\{a_1,a_2,c\}$ is a block of $C$.  Now consider $D=B^\ast\cup C$ as a free amalgam of $B^\ast$ and $C$ over $A^\ast$, that is,  with the blocks in $B^\ast$ and in $C$ and no others.  Then $D$ is a partial STS. We claim that $D\in \Cun$, $C\leq D$ and $B^\ast\leq^n D$.
		
First we show that $D$ has an HF-ordering where $C$ is an initial segment. This proves that $D\in\Cun$ and $C\leq D$.  Take an HF-ordering $<_C$ of $C$ and an HF-ordering $<_{B^\ast}$ of  $B^\ast=A^\ast\cup (B^\ast\smallsetminus C)$ where $A^\ast$ is an initial segment. Then the sum 
$$(C,<_C)+(B^\ast\smallsetminus C,<_{B^\ast}\restriction (B^\ast\smallsetminus C))$$ 
is an HF-ordering of $D$. If not, there is an element of $B^\ast\smallsetminus C$ that belongs to different blocks containing smaller elements, but these elements then are in $A^\ast\cup (B^\ast\smallsetminus C)$, contradicting the choice of the ordering $<_{B^\ast}$ of $B^\ast$ as an HF-ordering over $A^\ast$.
		
		Finally, we check that $B^\ast\leq^n D$.  Let $X\subseteq D\smallsetminus B^\ast = C\smallsetminus A^\ast$  with $|X|\leq n$; we show that $B^\ast\leq B^\ast \cup X$.  Since  $A^\ast\leq B^\ast$, we have an HF-ordering $<_{B^\ast}$ of $B^\ast$ with initial segment $A^\ast$.  Since $A^\ast\leq^n C$, we know that $A^\ast\leq A^\ast\cup X$, and hence there is an HF-ordering $<_{A^\ast \cup X}$ of $A^\ast\cup  X$ with initial segment $A^\ast$. We combine these orderings by taking the sum $(B^\ast,<_{B^\ast})+(X,<_{A^\ast\cup X}\restriction X)$. Since an element of $X$ cannot have a block containing elements of $B^\ast$ unless the elements belong to $A^\ast$, we obtain an HF-ordering of $B^\ast\cup X$ with initial segment $B^\ast$.
	\end{proof}

    \begin{prop}\label{amalgamation} $(\Cun,\leq)$ has the amalgamation property: if $A,B,C\in\Cun$, $A\leq B$ and $A\leq C$, there are $D\in\Cun$ and embeddings $f:B\rightarrow D$,  $g: C\rightarrow D$ such that  
    $$f\restriction A= g\restriction A, \quad f(B)\leq D \, \mbox{ and } \,  g(C)\leq D \, .$$
	\end{prop}
	\begin{proof}
		The proof of Proposition~\ref{Feb_23_2024_2} can be used with the same construction of  the partial Steiner triple system $B^\ast\cong_AB$ and the partial Steiner triple system $D$.
	\end{proof}

	\begin{prop}\label{Feb_23_2024_4} The generic is isomorphic to the free Steiner triple system with $\omega$ generators.
	\end{prop}
	\begin{proof} First note that the generic $F$ is an STS, that is, the product of two distinct elements $a_1,a_2$ is defined: let $b$ be a new vertex and let $A=\{a_1,a_2,b\}$ be an STS (so $ \{a_1, a_2, b \}$ is a block). Then $A\in\Cun$.  Let $B \leq F$ be such that $  \{a_1, a_2\} \subseteq B$. If $a_1 \cdot a_2 \notin B$, then $a_1 \cdot a_2$ can be put on top in an HF-ordering of $B$ and, therefore, $B \leq B(a_1 \cdot a_2)$. So there is an embedding $f:B(a_1 \cdot a_2)\rightarrow F$ over $B \supseteq\{a_1,a_2\}$. Moreover, $F$ must be a free Steiner triple system, since it has a well-founded HF-ordering.  On the other hand, the generic can be constructed by  including at each stage an element which does not form a block with any previous elements. Then the corresponding  HF-ordering has infinitely many elements which do not form a block with any preceding elements. By Proposition~\ref{May12_2.2}, these elements belong to a free base of~$F$.
	
	\end{proof}
	
	We extend the notion of strong subset so that it applies to infinite partial Steiner triple systems. 
\begin{df}  Let $A,B$ be arbitrary partial STSs.
	We write  $A\leq B$ if there is an HF-ordering $<$ of $B$ over $A$, or, equivalently, if there is an HF-ordering of $B$
	that has $A$ as an initial segment.  
\end{df}

	\begin{rmk}\label{Feb_23_2024_6.1} When $A$ is finite and $B$ is unconfined, this definition of $A\leq B$  is equivalent  to the definition given after Lemma~\ref{diamond}, that is  $A\leq B$ if and only if $A\subseteq B$ and $A\leq C$  for every finite $C$ such that  $A\subseteq C\subseteq B$.  More generally, if $A,B$ are arbitrary partial STSs and $B$ is unconfined, then
		$$A\leq B \mbox{ if and only if } A\subseteq B \mbox{ and } A\cap C \leq C \mbox{ for every finite } C\subseteq B.$$
	\end{rmk}
	\begin{proof}  Assume $A \subseteq B$ and $A\cap C\leq C$  for every finite $C\subseteq B$. We show that there is an HF-ordering of $B$  over $A$.  The argument is an application of the compactness theorem for propositional logic.  For every pair $(a,b)\in B^2$  we introduce a propositional variable $p_{(a,b)}$ meaning $a<b$. The set of sentences 
		$$\{\neg p_{(a,a)}\mid a\in B\}\cup\{p_{(a,b)}\wedge p_{(b,c)}\rightarrow p_{(a,c)}\mid a,b,c\in B\}\cup\{p_{(a,b)}\vee p_{(b,a)}\mid a,b\in B,\; a\neq b\}$$
		defines a linear ordering in $B$ and the set 
		$$\{\neg( p_{(b,a)}\wedge p_{(c,a)}\wedge p_{(b^\prime,a)}\wedge p_{(c^\prime,a)})\mid \{a,b,c\}\neq \{a,b^\prime,c^\prime\} \mbox{ blocks  of } B\}$$
		expresses that  the order is an HF-ordering of $B$. Finally, 
		$$\{ p_{(a,b)}\mid   a\in A,\; b\in B\smallsetminus A\}$$
		expresses that $A$ is an initial segment. By assumption, every finite subset of this set of sentences is satisfiable. By compactness, the whole set is satisfiable.
	\end{proof}

	\begin{prop} \label{Feb_23_2024_8} If  $M$ is an unconfined STS  and $A=\acl(A)\subseteq M$, then  $A\leq M$.
	\end{prop}
	\begin{proof}   We can expand the monster model $\mons$ by adding an HF-ordering $<$ of $\mons$, but we work in $\mons$.   
		By P.M. Neumann's Lemma (2.3 in \cite{pmneumann} --- also \cite[Corollary~4.2.2]{hodges}),  given  $a_1,\ldots,a_n\not\in \acl(A)$, for any set $B$ in $\mons$  there are $b_1\ldots b_n$ with $b_1,\ldots,b_n\equiv_A a_1\ldots a_n$ such that  $b_1,\ldots,b_n\not\in B$. Hence if $M\smallsetminus A=\{a_i\mid i<\kappa\}$  and  $p(x_i\mid i<\kappa)=\tp((a_i\mid i<\kappa)/A)$, then,  for any set $B$,  the type~$p$ has a realization  $(b_i\mid i<\kappa) $ such that   $\{b_i\mid i<\kappa\}\cap B=\emptyset$. Using this, for any $\lambda$  we can construct a sequence of models $(M_i\mid i<\lambda)$  containing $A$ such that  $M=M_0$, $M_i\cong_A M_j$  and $M_i\cap M_j=A$  for every $i\neq j$. Let $B=\cl_<(A)\subseteq \mons$ and let  $\mu= |B|= |A|+\omega$.  We may assume $\mu <\lambda$, which implies  $M_i\cap B = A$  for some $i$ (if $A\subsetneq M_i\cap B$ for every $i<\lambda$, we can choose $a_i\in (M_i\cap B)\smallsetminus A$ and the $a_i$ are all distinct, which contradicts $|B|<\lambda$).  By Lemma~\ref{Feb_23_2024_7}, $<$ is an HF-ordering over $B$. Therefore $<_i=<\restriction M_i$ is an HF-ordering of $M_i$ over  $A= B\cap M_i$. Since $M_i\cong_A M$,  there is an HF-ordering of $M$ over $A$.
	\end{proof}
	
	\begin{cor}\label{Feb_23_2024_9}  If $M_1\preccurlyeq M_2$  are unconfined  STSs, then  $M_1\leq M_2$.
	\end{cor}
	\begin{proof}  By~\ref{Feb_23_2024_8}  since  $\acl(M_1)=M_1$.
	\end{proof}

	\section{Binary trees} \label{section6}
	
	 This section introduces certain tree constructions in unconfined Steiner triple systems. These trees are an essential tool for our proof of the completeness of the theory of infinite unconfined Steiner triple systems in Section~\ref{section7}.

Let $M$ be an unconfined STS. A  \textbf{binary  tree of height $\alpha\leq \omega$} is a subset  $\{a_s\mid s\in 2^{<\alpha}\}$ of $M$ such that
\begin{enumerate}
	\item if $s\neq t$, then $a_s\neq a_t$
	\item $a_s= a_{s0}\cdot a_{s1}$
	\item these are all the blocks of the tree:  the only block containing $a_\emptyset$ is $\{a_\emptyset, a_0, a_1\}$  and if $s=ti$,  the only blocks containing $a_s$  are   $\{a_t,a_{t0},a_{t1}\}$  and  $\{a_s,a_{s0},a_{s1}\}$.   
\end{enumerate}
A \textbf{leaf} is an element $a_s$ of the tree such that there are no $a_{s0}, a_{s1}$ in the tree with $a_s=a_{s0}\cdot a_{s1}$.

\begin{rmk}\begin{enumerate}\label{tree1}
		\item A binary tree of height $n+1$ can be seen as a construction of $a_\emptyset$ from $\{a_s\mid s\in 2^n\}$ in the form of a binary tree, with the root $a_\emptyset$ at the top. Every free STS embeds many copies of this tree: it suffices to choose a level with at least  $2^n$ elements and take the products of these elements that belong to higher levels, thus obtaining the root $a_\emptyset$ after $n$ steps. If $\{a_s\mid s\in 2^n\}$ are the chosen elements in the given level, the tree can be described as $\{a_s\mid s\in 2^{\leq n}\}$, where  $a_s= a_{s0}\cdot a_{s1}$ if  $s\in 2^{<n}$.   A binary tree of height $\omega$  can also be seen as a (non well-founded) construction of $a_\emptyset$ over $\emptyset$. 
		\item The tree $\{a_s\mid s\in 2^{<\alpha}\}$  admits a  natural HF-ordering  that in the well-founded case reflects the order of construction of the root described in part 1: we order every level arbitrarily and we make the elements of level $i$  larger than the elements of level $i+1$. In particular, $a_s > a_t$  if $s\subsetneq t$.
	\end{enumerate}
\end{rmk}
An alternative HF-ordering can be defined on a binary tree, which will be used in the proof of Proposition~\ref{May_26_2024_1} below.
\begin{lemma}\label{April_11_2024_2} The binary tree of height $\alpha$ has an HF-ordering where the root  $a_\emptyset$ is the smallest element instead of being the largest.
\end{lemma}
\begin{proof}  If $<$ is  a natural HF-ordering of the binary tree, with $a_\emptyset$ at the top, then the reverse order $>$ is also an HF-ordering and now $a_\emptyset$ is the smallest element.
\end{proof}

Given an element $a$ of an infinite unconfined Steiner triple system, we can build a binary tree of arbitrary finite length having $a$ as its root. If the STS is $\omega$-saturated, we can find such a tree of height $\omega$.

\begin{lemma}\label{tree2}
	\begin{enumerate}
		\item An HF-ordering of an infinite STS does not have a maximum.
		\item If $<$ is an HF-ordering of an infinite Steiner triple system $M$, then for every $a\in M$, for every  $b \geq a$, there are $b_1,b_2>b$ such that  $a=b_1\cdot b_2$.
	\end{enumerate}
\end{lemma}
\begin{proof}
	1. Assume $a\in M$ is a maximum. Since $|M|>3$, the element $a$ belongs to two different blocks, $\{a,a_1,a\cdot a_1\}$ and $\{a,a_2,a\cdot a_2\}$.  By maximality, $a> a_1,a_1\cdot a, a_2, a_2\cdot a$, in contradiction with $<$ being an HF-ordering.
	
	2. Assume not, and take $b \geq a$ such that  if $a=a_1\cdot a_2$, then $a_1\leq b$ or $a_2\leq b$. Since $M$ is infinite, by part 1 we can choose $b_1>b_2>b$ such that  $b_1\cdot b_2<b_1$. Note that $\{b_1,a,a\cdot b_1\}\neq \{b_1,b_2, b_1\cdot b_2 \}$, since otherwise $a=b_1\cdot b_2$ and  $b_1,b_2>b$. If $b_ 1> a\cdot b_1$, then  the previous blocks contradict that $<$ is an HF-ordering. Hence, $b_1<a\cdot b_1$. In this case, $a= b_1\cdot (a\cdot b_1)$  with $b_1,a\cdot b_1>b$, a contradiction with the choice of $b$.
\end{proof}

\begin{prop}\label{tree3} Let $<$ be an HF-ordering of the infinite Steiner triple system $M$ and let $n<\omega$. For every $a\in M$    there is a binary tree $T=\{a_s\mid s\in 2^{<n}\}$  in $M$  such that
	\begin{enumerate}
		\item $a_s>a$  for every $s$
		\item $a_s>a_t$  if  $s\subsetneq t$.
	\end{enumerate}
\end{prop}
\begin{proof}

	We prove by induction on $n$ that  for every $a\in M$ for every finite $A\subseteq M$, there is a binary tree $T=\{a_s\mid s\in 2^{<n}\}$  in $M$  satisfying conditions 1,2 and such that there is no block connecting elements of $T$ with elements of $A$ (that is, no products $t \cdot a$ with $t \in T$ and $a \in A$ are in $T \cup A$).
	 The initial cases $n=0,1$ are trivial. Consider the case  $n=2$. Let $A^\prime$ be the set of products $\{x\cdot y\mid x,y\in A\}$, and let $b>a$ and larger than every element of $A^\prime$. We use Lemma~\ref{tree2} to get some $b_1 >b_2 >b$ in $M$ such that $b=b_1\cdot b_2$. We let $a_\emptyset = b_1$,  $a_0=b_2$ and $a_1 = b$.

	Now we assume inductively the result holds for $n\geq 2$.
	Given $a\in M$ and a finite set $A\subseteq M$, the inductive hypothesis gives some tree $T_0=\{a_s\mid s\in 2^{<n}\}$ as required for $a,A$.  Let $A^\prime =A\cup T_0$. Again by inductive hypothesis, we can find a tree $T_1=\{a^\prime_s\mid s\in 2^{<n}\}$ in $M$  with $a^\prime_s$ larger than every element of $A^\prime$, with no blocks connecting $T_1$ and $A^\prime$  and such that condition 2 holds. Now we combine $T_0$ and $T_1$ to get a binary tree $T=\{b_s \mid s\in 2^{<n+1}\}$. We set $b_{0s}= a_s$, $b_{1s}= a^\prime_s$  and  $b_\emptyset = a_\emptyset\cdot a^\prime_\emptyset$.  We first  check that   $a_\emptyset\cdot a^\prime_\emptyset > a^\prime_\emptyset$. Since $n\geq 2$ this is clear, since otherwise the blocks $\{a^\prime_\emptyset,a^\prime_0,a^\prime_1\}$  and  $\{a^\prime_\emptyset, a_\emptyset, a_\emptyset \cdot a^\prime_\emptyset\}$ are different and contradict that $<$ is an HF-ordering. Since there is no block connecting elements of $T_1$ with elements of $T_0$, $T$ is a tree.
\end{proof}

\begin{cor}\label{tree4}
	Let $<$ be an HF-ordering of the infinite STS $M$ and assume 
	$M$ is $\text{$\omega$-saturated}$ as an $(R, <)$-structure.  For every $a\in M$    there is a binary tree $T=\{a_s\mid s\in 2^{<\omega}\}$  in $M$ such that
	\begin{enumerate}
		\item $a_s>a$  for every $s$
		\item $a_s>a_t$  if  $s\subsetneq t$.
	\end{enumerate}
\end{cor}
\begin{proof}
	The tree is a realization of a type over $a$ in $M$ in the extended language with the order relation. The consistency of the type is granted by Proposition~\ref{tree3}.
\end{proof}

 Binary trees allow us to give conditions under which a certain element in an infinite unconfined STS $M$ can be placed between a subsytem $A$ of $M$ and the complement $M \smallsetminus A$. In the presence of a suitable amount of saturation, for a $<$-closed set $A$ we can always find an HF-ordering and a least element in $M \setminus A$ that is above $A$. This result in essential in the back-and-forth argument that proves the completeness of the theory of infinite unconfined STSs in Section~\ref{section7}.

 In the proposition below, the symbol $<$ denotes an HF-ordering, while $\leq$ is used for the relation of being a strong substructure.
\begin{prop} \label{May_26_2024_1} Assume  $M$ is an infinite unconfined STS, $A=\langle A\rangle \subseteq  M$  and  $a\in M\smallsetminus A$. The following are equivalent:
	\begin{enumerate}
		\item  $\cl_<(a)=\{a\}$ for some HF-ordering of $M$ with initial segment  $A$
		\item   $A\leq Aa\leq M$
		\item There are $M^\prime\succeq M$, an HF-ordering $<^\prime$ of $M^\prime$ and a binary tree $T=\{a_s\mid s\in 2^{<\omega}\}$ in $M^\prime$ such that $\cl_{<^\prime}(A)=A$,  $a=a_\emptyset$, $T>^\prime A$ and  $a_s>^\prime a_{s0},a_{s1}$ for every~$s\in$~$ 2^{<\omega}$.
	\end{enumerate}
	Moreover, these conditions imply $a\leq M$.
\end{prop}
\begin{proof} 1 $\Rightarrow$ (2 and $a\leq M$). Fix an HF-ordering $<$ of $M$ with initial segment $A$ and such that $\cl_<(a)=\{a\}$.  Then  we can move $a$ to the bottom (showing $a\leq M$) and we can move $a$ to the top of $A$, making it the first element of $M\smallsetminus A$ (showing $Aa\leq M$).
	
	2 $\Rightarrow$ 1. Since $Aa\leq M$ and $A\leq Aa$, there is an HF-ordering of $M$ having an initial segment $Aa$, where $a$ is the last element of the segment. Since $a\not\in A$,  $\cl_<(a)=\{a\}$.
	
	1 $\Rightarrow$ 3. Let $<$ be an HF-ordering of $M$ with initial segment $A$ such that $\cl_<(a)=\{a\}$.  Let $(M^\prime,<^\prime)\succeq (M,<)$ be $\omega$-saturated. Then $<^\prime$ is an HF-ordering of $M^\prime$, $\cl_{<^\prime}(A)=A$,  $a>^\prime A$ and  $\cl_{<^\prime}(a)=\{a\}$.   By Corollary~\ref{tree4}, choose a tree $T_0=\{a_s\mid s\in 2^{<\omega}\}$ in  $M^\prime$ such that $a_s>^\prime a$  for every $s$ and such that
	$a_s>^\prime a_t$  if  $s\subsetneq t$. Then $a\cdot a_\emptyset>^\prime a_\emptyset$, since otherwise the blocks $\{a_\emptyset,a_0,a_1\}\neq \{a_\emptyset, a, a\cdot a_\emptyset\}$ contradict that $<^\prime$ is an HF-ordering. Let $b_0=a\cdot a_\emptyset$ and $B_0= \{b_0\}\cup T_0$.  Note that $\cl_{<^\prime}(B_0)=B_0\cup\{a\}$. Now we modify the order $<^\prime$, thus obtaining a new HF-ordering $<_0$ such that $<_0\restriction B_0= <^\prime \restriction B_0$ and $ <_0\restriction (M\smallsetminus B_0)= <^\prime\restriction (M\smallsetminus B_0)$. This ordering $<_0$ is obtained by moving  $B_0$  below $a$ but above $A$, in particular by declaring that $a>_0 b_0$ and $B_0>_0 A$ and, hence,  $a>_0a_s$ for every $s$. This can be done by Lemma~\ref{June_25_2024_1.1}. We have obtained an incomplete tree with root $a$ followed by $b_0$ on the left and the complete tree $T_0$ on the right. Now $b_0$ is a leaf and we should complete the tree below $b_0$. The construction must be iterated $\omega$ times in order to obtain a full binary tree.
	
	3 $\Rightarrow$ 1. Let $M^\prime$, $<^\prime$ and $T$ be as in the assumption. Notice that $\cl_{<^\prime}(T)=T$. By Lemma~\ref{April_11_2024_2} $T$ has an HF-ordering $<^{\prime\prime}$ where $a_\emptyset$ is the smallest element. By Lemma~\ref{June_25_2024_1.1} we can assume that $<^{\prime\prime} = <^\prime \restriction T$ and $A$ is an initial segment. Let $<=<^\prime\restriction M$. Then $<$ is an HF-ordering of $M$, $A$ is an initial segment and $\cl_<(a)=\{a\}$.
\end{proof}

\begin{cor}\label{June_25_2024_4}  Let $M$ be  an infinite unconfined STS  and $A\subseteq M$.  If $<$ is an HF-ordering of $M$ such that $(M,<)$ is  $\kappa$-saturated for some $\kappa>|A|$ and $\cl_<(A)=A$, then there is some $a\in M\smallsetminus A$  such that  $Aa\leq M$.
\end{cor}
\begin{proof}   By $\kappa$-saturation there is some $b\in M$  such that $b>A$. By  Corollary~\ref{tree4} we find a binary tree $T>b$ of height $\omega$ in $M$ and by the proof of $ 3 \Rightarrow 2$ in Proposition~\ref{May_26_2024_1}, the root $a$ of the tree verifies $Aa\leq M$.
\end{proof}

	\section{Completeness and stability} \label{section7}
	
 In this section we show that the theory of infinite unconfined Steiner triple systems is complete and stable.  Since all free STSs are models of the theory, it follows that any two free STSs are elementarily equivalent. Moreover, there is a countable model of the theory which is not free.

\begin{lemma} \label{May_26_2024_2} Let $M$ be an unconfined STS.  If $A\leq M$ (as in Definition~\ref{def5_1}), then  $\langle A\rangle \leq M$ and there is an HF-ordering of $\langle A\rangle$ over $A$ whose final segment $\langle A\rangle \smallsetminus A$ is well-ordered. If $N$ is an unconfined STS and $B\leq N$ and $f:A\cong B$ is an isomorphism of partial STSs, then $f$ can be extended to an isomorphism $f^\prime :\langle A\rangle \cong \langle B \rangle$.
\end{lemma}
\begin{proof} We could use Lemma~\ref{order_closure} to verify the first assertion, but the following proof also gives the isomorphism. Decompose $\langle A\rangle\smallsetminus A$ into a disjoint union $\bigcup_{n<\omega}A_n$,  where $A_0$ is the set of elements of $\langle A\rangle\smallsetminus A$ which are a product of two elements of $A$,  and $A_{n+1}$ is the set of elements of  $\langle A\rangle\smallsetminus (A\cup A_0\cup \ldots\cup A_n)$  which are a product of two elements of  $A\cup A_0\cup\ldots\cup A_n$. Fix an HF-ordering of $M$ with initial segment $A$. There are no blocks between elements of $A_0$, and the only blocks of elements of $A_0$ with smaller elements are blocks with elements of $A$ that justify their inclusion in $A_0$. Hence we can move  $A_0$ down, just above $A$,  and we can give $A_0$ a well-founded HF-ordering over $A$. This shows that $A\cup A_0\leq M$. By a similar argument we can move down each $A_{n+1}$ above $A\cup A_0\cup \ldots \cup A_n$ and give it a well-founded HF-ordering over this set. The sequence of these well-founded orders is a well-founded HF-order of $\langle A\rangle \smallsetminus A$  which makes  $\langle A\rangle$ an initial segment of the order of $M$. If we perform a similar decomposition $\bigcup_{n<\omega}B_n$ of $\langle B\rangle\smallsetminus B$, then we can  extend $f$ inductively to an isomorphism of partial Steiner triple systems  $f_n : A\cup A_0 \cup \ldots \cup A_n\cong B\cup B_0\cup\ldots \cup B_n$. The union of this ascending chain of isomorphisms is the required isomorphism $f^\prime$.
\end{proof}

The following is a stronger form of Corollary~\ref{June_25_2024_4}.

\begin{lemma}\label{July_11_2024_1}  Let $M$ be an infinite unconfined  STS and $A\subseteq M$. If  $<$ is an HF-ordering of $M$, $\cl_<(A)=A$ and  $M$ is $\kappa$-saturated as an $\{R,<\}$-structure for some $\kappa>|A|$,  then  $Aa\leq M$ for some $a\in M\smallsetminus\langle A\rangle$.
\end{lemma}
\begin{proof} Since $\cl_<(A)=A$, we have $A\leq M$ by Lemma~\ref{Feb_23_2024_7}. Then $\langle A\rangle \leq M$ by Lemma~\ref{May_26_2024_2}.  By Corollary~\ref{June_25_2024_4} we know that there is $a\in M\smallsetminus \langle A\rangle$  such that  $\langle A\rangle a\leq M$. It is easy to see that  $Aa\leq \langle A\rangle a$ and hence  $Aa\leq M$.
\end{proof}

\begin{rmk}\label{May_26_2024_3} If $M$ is a saturated  model of cardinality $\kappa$ in the language $L$, then for every extension $L^\prime\supseteq L$ with $|L^\prime|\leq \kappa$, for  every set of sentences $\Sigma$ in $L^\prime$ consistent with $\Th(M)$, there is an expansion $M^\prime$ of $M$ such that $M^\prime\models \Sigma$. If $A\subseteq M$ has  cardinality $<\kappa$ we may allow $\Sigma$ to have parameters in $A$, since the expansion $M_A=(M,a)_{a\in A}$ is still saturated. Moreover, we can let $\Sigma$ contain free variables, since we can replace them by constants. The same happens if $M$ is a special model. For details see~\cite{casanovas}.
\end{rmk}

For simplicity, the next  two results are  formulated in terms of saturated models, but we could alternatively use special models ($\kappa$-saturated in the lemma and $\omega_1$-saturated in the theorem), whose existence can always be assumed without extra set-theoretical hypotheses. 

\begin{lemma}\label{May_2025}
	Let $M$ be a saturated unconfined STS, let $A\leq M$ and assume $|A|<\kappa =|M|$. Then there is an HF-ordering $<$ of $M$ such that $M$ is saturated as an $\{R,<\}$-structure and $\cl_<(A)=A$.
\end{lemma}
\begin{proof}
	Let $<$ be an HF-ordering of $M$ having $A$ as an initial segment. Notice that the expansion $M_A= (M,a)_{a\in A}$ is saturated. Let $(M^\prime_A,<^\prime)$ be a saturated model that is elementarily equivalent to $(M_A,<)$ and has cardinality $\kappa$. For each $a\in A$ that is not a product in $M$ of two elements $a_1,a_2<a$, the sentence $\forall xy \, (x,y<a\rightarrow x\cdot y \neq a)$ holds in $(M_A,<)$. Hence, $\cl_{<^\prime}(A)=A$. Since $M_A\cong M^\prime_A$, there is some ordering $<^{\prime\prime}$ of $M$ such that $(M_A,<^{\prime\prime})\cong (M^\prime_A,<^\prime)$. Then $<^{\prime\prime}$ satisfies the requirements.
\end{proof}

\begin{thm} \label{May_26_2024_4} Let $M,N$ be uncountable saturated unconfined STSs. The set of isomorphisms between countable substructures $A\subseteq M$ and $B\subseteq N$  (in particular,  $A,B$ are STSs)  such that $A\leq M$ and $B\leq N$ is a back-and-forth system between $M$ and $N$.
\end{thm}	
\begin{proof} 
	
 The  case $A=B=\emptyset$ is allowed and shows that there is at least one such isomorphism.  Now assume $f:A\cong B$, where $A\leq M$ and $B\leq N$ are  countable substructures, and let $a\in M\smallsetminus A$. We want to find countable substructures $A^\prime \leq M$ and $B^\prime \leq N$ extending $A$ and $B$ with $a \in A^\prime$  and $f^\prime: A^\prime\cong B^\prime$ for some $f^\prime\supseteq f$. The back step (i.e. extending $f$ by adding some $b\in N\smallsetminus B$) is similar.   We fix an arbitrary HF-ordering $<_M$ of $M$  where $A$ is an initial segment and we use Lemma~\ref{May_2025} to choose an HF-ordering $<_N$ of $N$ such that $(N,<_N)$ is saturated and $\cl_{<_N}(B)=B$. 
	
	\emph{Case 1}.  $\cl_{<_M}(a)=\{a\}$, or, equivalently,  $Aa\leq M$. Corollary~\ref{June_25_2024_4} gives $b\in N\smallsetminus B$ such that  $Bb\leq N$, in which case, by Proposition~\ref{May_26_2024_1} and  by changing the ordering $<_N$ if necessary,   we may assume that $\cl_{<_N}(b)=\{b\}$  and $b>_N B$.  Then $f$ extends to an isomorphism of partial Steiner triple systems $f^\prime:Aa\cong Bb$. By Lemma~\ref{May_26_2024_2},   we have $\langle Aa\rangle\leq M$, $\langle Bb\rangle\leq  N$ and  $f^\prime$ extends to an isomorphism  $f^{\prime\prime}:\langle Aa\rangle \cong \langle Bb\rangle$.
	
	\emph{Case 2}.  $\cl_{<_M}(a)\smallsetminus A=\{a_0,\ldots,a_n\}$  with  $a=a_0>_M a_1 >_M \ldots >_M a_n$.  Note that  $\cl_{<_M}(a_n)=\{a_n\}$  and $a_n\not\in A$. As in Case 1, we can find $b_n\in N\smallsetminus B$ such that $Bb_n\leq N$, and $f$ extends to an isomorphism $f_n:\langle Aa_n\rangle\cong \langle Bb_n\rangle$ such that $f_n(a_n)=b_n$.  Assume inductively that we have found $b_{i+1},\ldots,b_n\in N$ and an isomorphism 
	$$f_{i+1}:\langle Aa_{i+1}\ldots a_n\rangle \cong \langle B b_{i+1}\ldots b_n\rangle$$
	with $f_{i+1}\supseteq f$ and such that  $f_{i+1}(a_j)= b_j$  for every $j$.  If $a_i\in \langle Aa_{i+1}\ldots a_n\rangle$, we let $f_i=f_{i+1}$ and $b_i= f_i(a_i)$.  If $a_i\not \in \langle Aa_{i+1}\ldots a_n\rangle$, we see that $\cl_{<_M}(a_i)= \{a_i\}$ and, as for $a_n$, we get $b_i\in N\smallsetminus  \langle Bb_{i+1}\ldots b_n\rangle$ such that $ \langle Bb_{i+1}\ldots b_n\rangle b_i\leq N$ and $f_{i+1}$ extends to an isomorphism  $f_i:\langle Aa_i\ldots a_n\rangle\cong \langle B b_i\ldots b_n\rangle$ with $f_i(a_i)=b_i$. Finally  we obtain $b_0,\ldots,b_n$  and an isomorphism $f_0:\langle Aa_0\ldots a_n\rangle \cong \langle Bb_0\ldots b_n\rangle$ extending $f$   with $f_0(a_i)=b_i$ for every~$i$. Moreover, $\langle Aa_0\ldots a_n\rangle\leq M$ and   $\langle Bb_0\ldots b_n\rangle\leq N$.  Since $ \langle B b_0\ldots b_n\rangle \leq N$, if necessary we can change the  order on $\langle Bb_0\ldots b_n\rangle$ by that induced by $f_0$ so that $b_0>_Nb_1>\ldots  >_N b_n$.
	
 \emph{Case 3}.  $\cl_{<_M}(a)\smallsetminus A= \{a_i\mid i<\omega\}$ is infinite.   We claim that   for every $n<\omega$ and $i_0,\ldots,i_n\in\omega$ such that $a_{i_0}>_M a_{i_1}>_M\ldots >_M a_{i_n}$ there are $b_0,b_1,\ldots, b_n\in N$  such  that
	\begin{enumerate}
		\item $B\leq Bb_n\leq Bb_{n-1}b_n\leq \ldots \leq  Bb_0\ldots b_n \leq N$.
		\item $ Aa_{i_0}\ldots a_{i_n}\cong Bb_0\ldots b_n$ as partial STSs by some isomorphism  $f_0$ extending $f$  such that $f_0(a_{i_j})=b_j$  for every  $j\leq n$.
		We cannot ensure that $\langle Aa_{i_0}\ldots a_{i_n}\rangle\cong \langle Bb_0\ldots b_n\rangle $, since $Aa_{i_0}\ldots a_{i_n}\leq M$ may not hold.
	\end{enumerate}
	
	The $b_j$ are obtained inductively, starting with $b_n$. We choose $b_n\in N\smallsetminus B$ so that  $Bb_n\leq N$. It is easy to see that  $Aa_{i_n}\cong Bb_n$  by an isomorphism extending $f$.  Assume inductively  $f_{j+1}:Aa_{i_{j+1}}\ldots a_{i_n}\cong Bb_{j+1}\ldots b_n$ with  $f\subseteq f_{j+1}$  and  $f_{j+1}(a_{i_k})=b_k$, and assume 
	$$B\leq Bb_n\leq \ldots \leq Bb_{j+1}\ldots b_n\leq N \,.$$
	There are several cases  for $a_{i_j}$.  If $a_{i_j}= a_{i_k}\cdot a_{i_l}$  with  $j<k<l$ and $k\leq n$, we take $b_j =b_k\cdot b_l$. If $a_{i_j} = a_{i_k}\cdot c$  with $j<k\leq n$  and $c\in A$, we put  $b_j=b_k\cdot f_{j+1}(c)$. In other cases,  we use Lemma~\ref{July_11_2024_1} to get $b_j\in N\smallsetminus Bb_{j+1}\ldots b_n$ such that $Bb_jb_{j+1}\ldots b_n\leq N$ and which is not a product of elements of $Bb_{j+1}\ldots b_n$.  In all cases,  $Bb_{j+1}\ldots b_n\leq Bb_jb_{j+1}\ldots b_n\leq  N$  and  $f_{j+1}$ extends to some $f_j: Aa_{i_j}\ldots a_{i_n}\cong Bb_j\ldots b_n$  with $f_j(a_{i_j})=b_j$.

	Now we define a set  $\Sigma$ of formulas over $B$ in  the language extended by the symbol $<$, with free variables $\{x_i\mid i<\omega\}$ and  expressing the following conditions:
	\begin{enumerate}
		\item $(x_i\mid i<\omega)$ realizes the diagram of $(a_i\mid i<\omega)$ over $A$ (as a partial STS) conjugated by the isomorphism $f$
		\item $<$ is an HF-ordering   and coincides with $<_N$ on $B$
		\item $x_i<x_j$  if  $a_i<_M a_j$
		\item if $b\in B$ is not a product of smaller elements of $B$ in the order $<_N$, then \\
		 $\forall xy \, (x,y<b\rightarrow x\cdot y\neq b)$
		\item if $a_i$ is not the product of smaller elements $a_j,  a_k$ nor the product of a smaller $a_j$ with an element of $A$, then  $\forall xy\, (x,y<x_i\rightarrow x\cdot y \neq x_i)$.
	\end{enumerate} Every finite subset of $\Sigma$ is satisfiable in $N$. More precisely, for every $n<\omega$   the order~$<_N$ satisfies all five conditions if we restrict $\Sigma$ to the variables $x_0,\ldots,x_n$ and we assign to these variables the sequence $b_0,\ldots,b_n$ described above for $a_{i_0}>_M \ldots >_M  a_{i_n}$,  where $i_0,\ldots,i_n$ enumerates  $\{0,\ldots,n\}$ according to the HF-ordering $<_M$ of the elements $a_0,\ldots,a_n$.
	 Note that we can interpret the order as an HF-ordering of $N$ that coincides with $<_N$ on $B$ and has $B$ as an initial segment, followed by $b_n,\ldots,b_0$ in ascending order.
	Now we use Remark~\ref{May_26_2024_3} to obtain an interpretation of $\Sigma$ in $N$ over $B$, namely an HF-ordering $<^\prime$ of $N$ and a sequence $(b_i\mid i <\omega)$ of elements of $N$ that satisfy all conditions on $\Sigma$ over $B$.  Then $f\cup\{(a_j,b_j)\mid j<\omega\}$ is an isomorphism of partial STSs between $A(a_j\mid j<\omega)$  and $B(b_j\mid j<\omega)$. Note that  $\cl_{<^\prime}(B\cup\{b_j\mid j<\omega\}) =B\cup \{b_j\mid j<\omega\}$  and, therefore,  $B(b_j\mid j<\omega)\leq N$. Hence $\langle A(a_j\mid j<\omega)\rangle \leq M$ and  $\langle B(b_j\mid j<\omega)\rangle \leq N$, and so, by Lemma~\ref{May_26_2024_2},  we can extend $f$ to an isomorphism $f^\prime:\langle A(a_j\mid j<\omega)\rangle \cong \langle B(b_j\mid j<\omega)\rangle$.

\end{proof}

\begin{cor}\label{elementary} Let $M,N$ be infinite unconfined STS. If $A\leq M$, $B\leq N$ and $f:A\rightarrow B$ is an isomorphism of partial STS, then $f$ is a partial elementary map from $M$ to $N$.
	\end{cor}
	\begin{proof} First consider the case where $A$ and $B$ are countable.  By Lemma~\ref{May_26_2024_2}, $\langle A\rangle \leq M$,  $\langle B\rangle \leq N$  and $f$ extends to some isomorphism  $f^\prime:\langle A\rangle \cong \langle B\rangle $.
		 We can find uncountable saturated (or special $\omega_1$-saturated) elementary extensions  $M^\prime\succeq M$ and $N^\prime\succeq N$. Notice that  $A\leq M^\prime$  and $B\leq N^\prime$.  By  Theorem~\ref{May_26_2024_4}, $f^\prime$ belongs to a back-and-forth system between $M^\prime$ and $N^\prime$ and so it is a partial elementary map from $M^\prime$ to $N^\prime$ and also from $M$ to $N$.
		
		Consider now the general case.  Fix an HF-ordering of $M$ having $A$ as an initial segment.  Consider a finite set $A_0\subseteq A$ and let $A_0^\prime =\cl_<(A_0)$.  Since $A$ is an initial segment, $A_0^\prime\subseteq A$. By Lemma~\ref{Feb_23_2024_7},  we may assume that $A_0^\prime$ is an initial segment of  $A$. Then $A_0^\prime\leq A$  and  $A_0^\prime\leq M$. Now order $B$ in such a way that $f$ is an order isomorphism. This gives an HF-ordering of $B$.  Moreover, if $B_0=f(A_0)$ and $B_0^\prime= f(A_0^\prime)$, then  $B_0\subseteq B_0^\prime$ and $B_0^\prime$ is an initial segment of $B$. Hence  $B_0^\prime\leq B$  and  $B_0^\prime\leq N$. Since $A_0^\prime$ and $B_0^\prime$ are countable, by the initial case we know that $f\restriction A_0^\prime$ is elementary. Hence, $f$ is elementary.
	\end{proof}

 The completeness of the theory of infinite unconfined STS is implicit in the proof of Corollary~\ref{elementary}.
	
	\begin{cor} \label{completeness} The theory of all infinite unconfined STS is complete; it is the complete theory of every free STS with $\geq 3$ generators.
	\end{cor}
	\begin{proof}  Corollary~\ref{elementary} implies that the empty mapping between unconfined STS is elementary. The claim follows directly.
	\end{proof}

	\begin{rmk} Assume $M$ is an infinite unconfined STS and  $A=\langle A\rangle \leq M$. By Corollary~\ref{elementary}, for any $a\not\in A$ the condition $Aa\leq M$ from Proposition~\ref{May_26_2024_1}
		defines a complete type $\tp(a/A)$: if $b\not\in A$ and $Ab\leq M$, the mapping which is the identity on $A$ and sends $a$ to $b$ is elementary. 
		\end{rmk}

	\begin{lemma}\label{elementary2}
		Let $M$ be an infinite unconfined STS and let $\bar{a},\bar{b}$ be tuples from $M$. If for some HF-orderings $<$ and $<^\prime$  of $M$,  $\cl_<(\bar{a})$ and $\cl_{<^\prime}(\bar{b})$ have the same diagram as partial STS, then  $\tp_M(\bar{a})=\tp_M(\bar{b})$.
	\end{lemma}
	\begin{proof} By Corollary~\ref{elementary}, since $\cl_<(\bar{a})\leq M$ and  $\cl_{<^\prime}(\bar{b})\leq M$.
		\end{proof}

	\begin{prop}\label{stable} The theory of free STSs  is stable.
	\end{prop}
	\begin{proof}  We count  $1$-types over a model $M$ and show that their number is $\leq |M|^\omega$. We assume $M$ is an elementary substructure of the monster model $\mons$. Then  $M\leq \mons$ by  Corollary~\ref{Feb_23_2024_9}, and we can fix an HF-ordering $<$ of $\mons$ where $M$ is an initial segment. If $a,b \in \mons$ and $\cl_<(a)$, $\cl_<(b)$ have the same diagram over $M$ as partial STS, then, by Lemma~\ref{elementary2}, they have same type over $M$.
	Hence, the number of $1$-types over $M$ is bounded by the number of partial STS isomorphism types of countable tuples over $M$, which is clearly $\leq |M|^\omega$.
	\end{proof}

	\begin{prop}\label{April_11_2024_6}
		Let $M,N$ be models of the theory of free STSs. If $M\leq N$, then  $M\preccurlyeq N$.
	\end{prop}
	\begin{proof}   
By Corollary~\ref{elementary}, the identity on $M$ is a partial elementary map from $N$ to $N$.
	\end{proof}
	
	\begin{prop}\label{acl} Let $M$ be an infinite unconfined STS and let $A =\langle A\rangle \subseteq M$ be infinite.  Then $A\equiv M$. Moreover,
		\begin{enumerate}
			\item $\acl(A)\preceq M$
			\item $A \leq M$ if and only if $ \acl(A)=A$.
		\end{enumerate}
	\end{prop}
	\begin{proof} Since $A$ is an infinite unconfined STS, by completeness of the theory $A\equiv M$.
	Claim 1 follows from Proposition~\ref{April_11_2024_6}, since, by Proposition~\ref{Feb_23_2024_8},  $\acl(A)\leq M$.
	Claim 2  follows from 1 and Proposition~\ref{April_11_2024_6}.
	\end{proof}
	
	\begin{rmk}
		Let $M$ be an infinite unconfined STS. If $A\subseteq M$, then $\acl(A)$ is the intersection of all substructures $B\supseteq A$ of $M$ such that $B\leq M$.
	\end{rmk}
	\begin{proof} The trivial case where  $A$ is contained in a block is easy. Otherwise, $\langle A\rangle$ is infinite and, by Proposition~\ref{acl}, $\acl(A)$ contains the intersection of all substructures $B$ of $M$  such that $A\subseteq B\leq M$. On the other hand, if $B$ is one of these structures, again by Proposition~\ref{acl}, $\acl(A)\subseteq \acl(B) =B$.


		\end{proof}

	\begin{rmk}\label{April_11_2024_4} Let $F_n$ be the free STS with $n\geq 3$ generators. Every surjective homomorphism $f:F_n\rightarrow F_n$  is an automorphism of $F_n$.
	\end{rmk}
	\begin{proof}  See~\cite{barcas2}, Corollary~2.11.
	\end{proof}
As remarked by the anonymous referee, a group where every surjective endomorphism is an automorphism is called \textit{Hopfian}. Proofs of this property for groups often use ad-hoc techniques, whereas our argument applies to any free algebras arising from a class that satisfies the hypotheses of Theorems~1 and~2 in \cite{jontarski} -- that is, the class contains a finite algebra with more than one element, and is such that any algebra that is free for the finite elements in the class is free for the whole class (for the second condition in the case of STSs see \cite[Lemma~2.6]{barcas2}).
	\begin{prop}\label{April_11_2024_5} Let $F_n$ be the free STS with $n\geq 3$ generators. Every elementary embedding $f:F_n\rightarrow F_n$ is surjective and, therefore, is an automorphism of $F_n$.
	\end{prop}
	\begin{proof} Let  $F_n=\langle \bar{a}\rangle$, where $\bar{a}=a_1,\ldots,a_n$ is a free base,  and let $a_i^\prime =f(a_i) = t_i(\bar{a})$  for every   $i=1,\ldots,n$, where $t_i$ is a corresponding term. Then  $F_n$ satisfies
		$$\exists x_1\ldots x_n  \bigwedge_{i=1}^n  a_i^\prime = t_i(x_1,\ldots,x_n).$$
		Since $f(F_n)\preceq F_n$, the same sentence holds in $f(F_n)$ and, by isomorphism,  $F_n$ satisfies
		$$\exists x_1\ldots x_n  \bigwedge_{i=1}^n  a_i= t_i(x_1,\ldots,x_n).$$
		Choose $\bar{b}= b_1,\ldots,b_n\in F_n$  such that   $a_i= t_i(\bar{b})$  for every $i=1,\ldots,n$. Then  $F_n=\langle b_1,\ldots,b_n\rangle$  and $b_1,\ldots,b_n$ are independent by Remark~\ref{April_11_2024_4}. Now consider the automorphism $g:F_n\rightarrow F_n$  defined by  $g(b_i)=a_i$  for  $i=1,\ldots,n$. Observe that  for every $i$
		$$g(a_i)= g(t_i(\bar{b}))= t_i(g(\bar{b}))=t_i(\bar{a})= a_i^\prime.$$
		This implies that $f=g$ and, therefore, that $f$ is an automorphism.
	\end{proof}

	\begin{prop}\label{April_11_2024_7} If $3\leq n<m$, there is no elementary embedding of $F_m$, the free STS with $m$ generators, into $F_n$, the free STS of $n$ generators, despite the elementary equivalence of $F_m$ and $F_n$.
	\end{prop}
	\begin{proof} Assume $n<m$ and $f:F_m\rightarrow F_n$ is an elementary embedding  with $f(F_m)=F_m^\prime$. There is a copy $F_m^{\prime \prime}$ of $F_m$ such that $F_n \leq F_m^{\prime \prime}$, so $F_n\preccurlyeq F_m^{\prime \prime}$ by Proposition~\ref{April_11_2024_6}. Without loss of generality, $F_m=F_m^{\prime \prime}$. Let $g$ be the inclusion map $F_n \mapsto F_m$ restricted to $F_m^\prime$. Then $ g \circ f\, : F_m \to F_m$ is an elementary injection with $g \circ f(F_m)\neq F_m$, contradicting Proposition~\ref{April_11_2024_5}.
	
	\end{proof}

	\begin{prop}   There is a (countable) unconfined STS which is not free.
	\end{prop}
	\begin{proof} Let $N=F_3$ be  the free STS generated by three elements  $a_1,a_2,a_3$, and let \\ $b = (a_1 \cdot a_2) \cdot (a_1 \cdot a_3)$. As shown in~\cite{barcas2},  Example~6.9, the homomorphism $f:N\rightarrow N$ determined by $f(a_1)=b$, $f(a_2)=a_2$ and $f(a_3)=a_3$ is an embedding but it is not surjective. Hence, there is some proper substructure of $N$ isomorphic to $N$. So we can define an ascending chain $(N_n\mid n<\omega)$ of STSs such that $N_n\cong N$ and $N_n\neq N_{n+1}$. Let $N_\omega=\bigcup_{n<\omega}N_n$. 
 Since $N_\omega$ is a union of a countable chain of unconfined STSs and unconfinedness fails on a finite subset, we have that $N_\omega$ is unconfined. Moreover, it is not finitely generated. Since it is countable, if it is free then it has a free base $B$ of $\omega$ elements and  it is (isomorphic to)  the generic. There is an embedding $f:F_4\rightarrow N_\omega$ such that  $f(F_4)\leq N_\omega$.  Choose $n$ such that the image in $f$ of  the base of $F_4$  is contained in $N_n$.  Then $f(F_4)\leq N_n$  and by Proposition~\ref{April_11_2024_6},  $f:F_4\rightarrow N_n$ is elementary. Since  $N_n\cong F_3$, this contradicts Proposition~\ref{April_11_2024_7}.
		Hence, $N_\omega$ is not free.
	\end{proof}


\subsubsection*{Acknowledgements} We are grateful to the anonymous referee for a prompt review and valuable suggestions and comments that have much improved the manuscript. The first author acknowledges membership of \textit{Gruppo Nazionale per le Strutture Algebriche, Geometriche e loro Applicazioni} (GNSAGA - INdAM).

\subsubsection*{Funding} Research partially supported by grants PID2020-116773GB-I00 and PID2023-147428NB-I00 from the Spanish Government. The first author has also been supported by PRIN2022 \textit{Models, sets and classifications}, prot. 2022TECZJA.

\end{document}